\documentclass[12pt,twoside]{article}

\usepackage[T1]{fontenc}
\usepackage[latin1]{inputenc}
\usepackage[english]{babel}
\usepackage[babel]{csquotes}

\usepackage{cite}

\usepackage{amssymb}
\usepackage{amsmath}
\usepackage{amsthm}
\usepackage{latexsym}
\usepackage{graphicx}
\usepackage{mathrsfs}
\usepackage{mathrsfs}

\usepackage[usenames,dvipsnames]{color}

\def\carlo #1{{\color{blue}#1}}
\def\carlo#1{#1}

\theoremstyle{plain}
\newtheorem{thm}{Theorem}[section]
\newtheorem{cor}[thm]{Corollary}
\newtheorem{lem}[thm]{Lemma}
\newtheorem{prop}[thm]{Proposition}
\newtheorem{defin}[thm]{Definition}
\theoremstyle{definition}
\newtheorem{rmk}[thm]{Remark}

\def\En{\mathbb{N}}

\def\Ar{\mathbb{R}}

\def\Pi{\mathbb{P}}
\def\f{\mathcal{F}}
\def\Z{\mathcal{Z}}
\def\exval{\mathbb{E}}
\def\cL{\mathscr{L}}
\def\eps{\varepsilon}

\def\beq{\begin{equation}}
\def\eeq{\end{equation}}

\def\rarr{\rightarrow}

\def\rarrw{\rightharpoonup}
\def\weakstar{\stackrel{*}{\rightharpoonup}}
\def\embed{\hookrightarrow}
\def\cembed{\stackrel{c}{\hookrightarrow}}

\def\l|{\left\|}
\def\r|{\right\|}

\def\V{\mathcal{V}}
\def\H{\mathcal{H}}
\def\A{\mathcal{A}}
\def\C{\mathcal{C}}
\def\B{\mathcal{B}}
\def\U{\mathcal{U}}
\def\W{\mathcal{W}}
\def\E{\mathbb{E}}
\def\P{\mathcal{P}}

\begin{document}
\begin{center}
{\huge\rm Singular Stochastic Allen-Cahn equations\\
with dynamic boundary conditions}
\\[0.5cm]
{\large\sc Carlo Orrieri$^{(1)}$}\\
{\normalsize e-mail: {\tt orrieri@mat.uniroma1.it}}\\[.25cm]
{\large\sc Luca Scarpa$^{(2)}$}\\
{\normalsize e-mail: {\tt luca.scarpa.15@ucl.ac.uk}}\\[.25cm]
$^{(1)}$
{\small Dipartimento di Matematica, Sapienza Universit\`a di Roma}\\
{\small Piazzale Aldo Moro 5, 00185 Roma, Italy}\\[.2cm]
$^{(2)}$
{\small Department of Mathematics, University College London}\\
{\small Gower Street, London WC1E 6BT, United Kingdom}\\[1cm]
\end{center}       

\begin{abstract}
We prove a well-posedness result for stochastic Allen-Cahn type equations in a bounded domain
coupled with generic boundary conditions. 
The (nonlinear) flux at the boundary aims
at describing the interactions with the hard walls and is motivated by some recent literature in physics. 
\carlo{The singular character of the drift part allows for a large class of 
maximal monotone operators, generalizing the usual double-well potentials.
One of the main novelties of the paper is the absence of any growth condition on the drift term of the evolution,
neither on the domain nor on the boundary.
A well-posedness result for variational solutions of the system is presented using
{\it a~priori} estimates as well as monotonicity and compactness techniques.}
A vanishing viscosity argument for the dynamic on the boundary is also presented.
\\[.5cm]
{\bf AMS Subject Classification:} 35K55, 35K61, 35R60, 60H15, 80A22.\\[.5cm]
{\bf Key words and phrases:} Allen-Cahn equation, dynamic boundary conditions,
						singular potentials, well-posedness, asymptotic estimates.
\end{abstract}

\pagestyle{myheadings}
\newcommand\testopari{\sc Carlo Orrieri and Luca Scarpa}
\newcommand\testodispari{\sc Singular Stochastic Allen-Cahn equations}
\markboth{\testodispari}{\testopari}


\thispagestyle{empty}

\section{Introduction}
\setcounter{equation}{0}
\label{intro}

Allen-Cahn type equations were introduced within the Van der Waals theory of phase transitions as a basic model 
to describe the evolution of a two-phase fluid. 
Generally, the unknown process is called (non-conserved) \textit{order parameter} and represents the 
normalized density of one of the two involved phases.
The starting point of the theory is the definition of a free energy functional associated to the order parameter 
$x(\cdot)$, which is given by 
\begin{equation}\label{free_en}
\mathcal{E}(x):= \int_D \left[  \frac{1}{2}| \nabla x |^2 + F(x) \right] dz.
\end{equation} 
In the classical setting $F: \Ar \to [0,+\infty)$ is a smooth double well potential representing 
the energy density (e.g. $F(x) = 1/4(x^2-1)^2$) which favours the pure states. 
The gradient part instead, taking into account the interactions at small scales, 
prevents instantaneous jumps between the two pure phases, penalizing the variation of $x$.
Then, the Allen-Cahn equation can be viewed as the $L^2$-gradient flow of 
the Ginzburg-Landau free energy \eqref{free_en}, i.e.~the semilinear parabolic PDE of the form
\begin{equation}\label{allen-cahn}
  \partial_tx - \Delta x + F'(x)=0\,.
\end{equation} 

In order to model the thermal fluctuation of the system it is quite natural to perturb the equation with a random force and,
from a phenomenological point of view, the choice of space time white noise seems vary natural. 
Unfortunately, nonlinear equations such as \eqref{allen-cahn} become ill-posed 
in space dimensions  $N \geq 2$ as soon as a white noise term is added. 
A classical way to bypass the problem is to smooth out the noise via a suitable covariance operator. 
Given $D\subseteq\Ar^N$ ($N\geq2$) a smooth bounded domain with smooth boundary $\Gamma$, 
what we end up with is the so called stochastic Allen-Cahn equation:  
\begin{equation}\label{allen-cahn-s}
dx_t - \Delta x_t\,dt - F'(x_t)\,dt =
  B(t, x_t)\,dW_t \qquad\text{in}\quad (0,T)\times D
\end{equation}  
with a given initial datum
\[
  x(0)=x_0 \quad\text{in}\quad D\,,
\]
where $T>0$ is a fixed final time,
$W$ is a cylindrical Wiener process on a separable Hilbert space $U$
and $B$ is a Hilbert-Schmidt
operator from $U$ to $L^2(D)$ depending on $x$ as well.

In the classical literature on the deterministic and stochastic
Allen-Cahn equation, the order parameter is usually assumed to
satisfy homogeneous Neumann conditions on the boundary $\Gamma$, 
which may be interpreted as a null interaction of the phase-transition
with the hard walls.
In this setting, well-posedness results for the stochastic Allen-Cahn equation
can be found e.g.~in the monograph \cite{daprato2014stochastic}, within the framework of dissipative SPDEs.

More recently, the class of energy functionals has been extended in several models
to take into account also a possible interaction of the phase transition phenomenon with the hard walls.
The main idea is to require that the free energy functional
could (possibly) penalize the variation of $x$ and the pure states also on the boundary $\Gamma$:
to this aim, if we denote with $\nabla_\Gamma$ the surface gradient on the boundary,
the form of the functional becomes 
\begin{equation}\label{free_en_boundary}
\mathcal{E}(x) := \int_D \left[  \frac{1}{2}| \nabla x |^2 + F(x) \right] dz 
+ \int_\Gamma \left[  \frac{\eps}{2}| \nabla_\Gamma y |^2 + F_\Gamma \right] d\zeta\,,
\end{equation}
where $\eps\geq0$ is fixed and 
$F_\Gamma:\Ar\rarr[0,+\infty)$ is another smooth double-well potential acting on the boundary.
Arguing as before, the $L^2$-gradient flow of this generalized free energy describes the following system
\begin{align*}
  \partial_tx - \Delta x+ F'(x) =0 \quad&\text{in } (0,T)\times D\\
  \partial_tx + \partial_{\bf n}x - \eps\Delta_{\Gamma}x+ F'_\Gamma(x) =0\quad&\text{in } (0,T)\times\Gamma\\
  x(0)=x_0 \quad&\text{in } D\\
  x(0)=x_0|_{\Gamma} \quad&\text{in } \Gamma\,,
\end{align*}
where
the symbol $\partial_{\bf n}$ denotes the outward normal derivative on $\Gamma$
and $\Delta_\Gamma$ is the usual Laplace-Beltrami operator.
Note that the presence of a free energy term on the boundary leads to non-standard dynamic conditions on the boundary 
(i.e.~that involve the time derivative of $x$ on $\Gamma$), in contrast with the classical homogeneous Neumann conditions for $x$. Let us point out that in the last few years there has been a lot of interest in
describing phase separation phenomena in
confined systems under general boundary conditions, see e.g. \cite{ball1990spinodal}, \cite{fischer1997novel}.

Finally, since the physical system may be subject also to a thermal fluctuation in the boundary  (due
both to the diffusion from the interior of $D$ and to a boundary noise) the natural idea is to
perturb the equation satisfied by $x$ on $\Gamma$ the same way that we have done for the one in the interior of $D$.
These considerations lead
to consider stochastic boundary conditions of the type
\beq\label{allen-cahn-sb}
  dx_t + \partial_{\bf n}x_t\,dt - \eps\Delta_{\Gamma}x_t\,dt + F_\Gamma'(x_t)\,dt =
  B_\Gamma(t, x_t)\,dW_t^\Gamma \quad\text{in } (0,T)\times\Gamma\,,
\eeq
where here $W^\Gamma$ is a cylindrical Wiener process on another separable Hilbert space $U_\Gamma$,
independent of $W$, and $B_\Gamma$ is a random time-dependent Hilbert-Schmidt
operator from $U_\Gamma$ to $L^2(\Gamma)$ of multiplicative type.

In the present paper we are interested in 
studying the stochastic system arising from the equations \eqref{allen-cahn-s} and \eqref{allen-cahn-sb}
from a more general mathematical perspective.
The main extension that we carry out concerns the form of the double-well potentials:
more precisely, instead of assuming that $F$ and $F_\Gamma$ are smooth functions on $\Ar$,
we simply require that $F=j+G$ and $F_\Gamma=j_\Gamma+G_\Gamma$,
where $j, j_\Gamma:\Ar \to [0,+\infty)$ are given convex functions with 
subdifferentials $\beta=\partial j$ and $\beta_\Gamma=\partial j_\Gamma$ everywhere defined, 
respectively, and $G, G_\Gamma$ are 
smooth functions with Lipschitz differentials $\pi, \pi_\Gamma$, respectively.
This means essentially that we are looking at the double-well potentials $F$ and $F_\Gamma$
as sufficiently smooth concave perturbations of convex potentials.
Bearing in mind these considerations, in the present paper we are concerned with the
following system:
\begin{align}
  \label{1}
  dx_t - \Delta x_t\,dt + \beta(x_t)\,dt + \pi(x_t)\,dt \ni
  B(t, x_t)\,dW_t \quad&\text{in } (0,T)\times D\\
  x=y \quad&\text{in } (0,T)\times\Gamma\\
  \label{2}
  dy_t + \partial_{\bf n}x_t\,dt - \eps\Delta_{\Gamma}y_t\,dt + \beta_\Gamma(y_t)\,dt + \pi_\Gamma(y_t)\,dt \ni
  B_\Gamma(t, y_t)\,dW_t^\Gamma \quad&\text{in } (0,T)\times\Gamma\\
  \label{init1}
  x(0)=x_0 \quad&\text{in } D\\
  \label{init2}
  y(0)=x_0|_{\Gamma} \quad&\text{in } \Gamma\,.
\end{align}
More specifically, we aim at proving well-posedness for problem \eqref{1}--\eqref{init2}
both in the case $\eps>0$ and $\eps=0$, as well as a suitable continuity of the solutions with respect to 
the parameter $\eps\geq0$, under no restrictive growth assumptions on the potentials.

Mathematical results concerning the Allen-Cahn (and similarly Cahn-Hillard) equation have been obtained recently
in the framework of generalized (deterministic) boundary dynamics. Let us mention e.g.~\cite{cal-colli,
col-gil-spr, grass, Liero2013bulk,
colli2015optimal}  and the references therein. On the contrary, not much is known on the stochastic counterpart, 
where one is interested in the dynamical impact of a noise term on the boundary. In this direction we have to mention
\cite{barbu2015stochastic, bonaccorsi2014variational}, where the authors study (nonlinear) diffusion problems 
with stochastic boundary conditions in a variational framework, and  \cite{yang2007impact} where long-time properties 
of the Cahn-Hillard equation are investigated.

\carlo{Concerning general well-posedness results for stochastic PDEs, 
let us mention \cite{gess2012strong} and the references therein, where unique existence 
of analytical strong solutions for a large class of SPDEs of gradient type is exhibited.
In that case, a crucial hypothesis used by the author is the sub-homogeneous character of the potential, 
which unfortunately forbids e.g. exponential growth.
In order to avoid any growth condition, in the present paper we only ask that $D(\beta) = \Ar$.
Still, this hypothesis does not seem to be optimal: it is not needed for the well-posedness of deterministic systems 
(see e.g. \cite{cal-colli}),
whereas in the stochastic formulation seems to be essential, at least in the approach we develop.}
Actually, in our case of interest $\beta = \partial j$ with $j$ being a convex potential, this restriction on the domain is the 
most general assumption in literature: it was considered for the first time in \cite{barbu}
in a problem related to existence of semilinear Laplace-driven stochastic equations and then in \cite{mar-scar}
when studying well-posedness for a class of abstract semilinear SPDEs with singular drift, avoiding any conditions 
on the growth of $\beta$. Again, let us remark that we are not able to consider a graph of the form 
$\tilde \beta = \partial I_{[-1,1]}$, but only an approximation of it defined everywhere in~$\Ar$.
With respect to the result obtained in \cite{mar-scar}, tailored for a large class of singular dissipative SPDEs, 
here we focus on a precise choice of diffusion operator, the Laplacian. This is motivated by the physical description 
of the model, but greater generality can be achieved without any substantial change in the proof.  
Within this framework, 
the key idea is to get good estimates in expectation and produce the pathwise counterpart in a set $\Omega'$
of probability $1$ using a suitable regularization on the noise.

The strategy of the proof is as follows.  We start by rewriting the system with additive noise as an equation for the pair
$(x,y)$ in the product space $L^2(D) \times L^2(\Gamma)$. Here we develop a variational approach \`a la Krylov,
Rozovski{\u\i} and Pardoux. In particular we define a suitable Gelfand triple and we smooth out the equation via Yosida
approximations of the singular part. In this way, the approximated system satisfies the usual assumptions and a version of the 
It\^o formula can be applied. We derive estimates in expectation of the solution as well as of the monotone maps. By 
compactness we pass to the limit pathwise to get a candidate limit equation and we identify the drift part as an element 
of the maximal monotone graphs $\beta$ and $\beta_\Gamma$. 
Then we recover uniqueness of the solution which is essential to infer 
measurability in $\omega$ of the limit. At last we generalize the well-posedness result also to noises of 
multiplicative type using a standard fixed-point argument.

A crucial step to get uniqueness of the solution is the application of the It\^o formula, 
for which a suitable smoothing of the equation is required. A classical way of proceeding is to apply
the resolvent operator of the diffusion to the equation itself. In our case of interest, rewriting the equation in the 
product space $L^2(D) \times L^2(\Gamma)$, the diffusion operator is not ``standard'' and has the form
\begin{equation}
\C_\eps(x,y) := (-\Delta x , \partial_{\bf n}x - \eps\Delta_\Gamma y),
\end{equation}        
where $y = \tau x$ is the trace part. Hence, given a couple of functions $(f,g)$, 
one needs to study the smoothing effect of the resolvent $(I + \delta\C_\eps)^{-1}$
through an \textit{ad hoc} regularity analysis 
of the associated elliptic system
\[\begin{cases}
  u-\Delta u = f \quad&\text{in } D\,,\\
  u=v\,, \quad u+\partial_{\bf n}u-\eps\Delta_\Gamma v = g \quad&\text{in } \Gamma\,.
\end{cases}
\]
Precisely, what we show is the ultracontractivity of $(I + \delta\C_\eps)^{-1}$ from  $L^1(D)\times L^1(\Gamma)$ to 
$L^\infty(D)\times L^\infty(\Gamma)$. To this aim, we generalize a classical regularity result 
by Stampacchia for elliptic equations with homogeneous boundary conditions
contained in \cite{stamp}
and subsequently prove a version of the maximum principle with 
data $(f,g) \in L^1(D)\times L^1(\Gamma)$. 
Let us note that the study of this operator forces us to impose some additional constraints 
on the relative growth of $\beta$ and $\beta_\Gamma$, which are indeed quite 
natural from the point of view of the physical applications.   
All the results mentioned above are collected in the Appendix.

Throughout the paper we use the parameter $\eps$ to indicate the presence of the diffusion operator 
$\eps \Delta_\Gamma$ on the boundary. The presence/absence of this term creates a gap between 
the effective domain of $\C_\eps$, when $\eps > 0$ and $\eps = 0$.
Form the physical point of view, a vanishing viscosity argument on the boundary becomes 
interesting as it is related to the formation of sharp interfaces between the two phases. 
What we show is the well-posedness of the problem in the singular case $\eps = 0$ as well 
as the continuous dependence of the solutions to the system \eqref{1}, \eqref{2} with respect to the variation of $\eps \geq 0$.

The paper is organised as follows: in the first section we introduce the notations, assumptions and 
we present the main results. Section 3 is devoted to proving the well-posedness of the system with additive noise:
here, we study the approximated equation and we pass to the limit using compactness arguments. 
In Sections 4 and 5 we extend the previous result to the Allen-Cahn equation with multiplicative noise and 
we study the asymptotic behaviour of the system as $\eps \to 0$. Finally, in the Appendix, 
we derive the smoothing properties of the diffusion operator $\C_\eps$.


\section{Notation, setting and main results}
\setcounter{equation}{0}
\label{results}

In the section we state the notation that we use and the precise assumptions of the work;
moreover, the concept of solution and the main results are presented.

\subsection{Notation}
Throughout the paper, $\left(\Omega, \f, \mathbb{F}, \Pi\right)$ is a filtered probability space,
with the filtration $\mathbb{F}=\left(\f_t\right)_{t\in[0,T]}$
satisfying the so-called "usual conditions" (i.e.~it is saturated and right continuous).
As we have anticipated, $D\subseteq\Ar^N$ is a smooth bounded domain with smooth boundary $\Gamma$
and $T>0$ is a fixed time.

If $U$ is a Banach space, for any $t\in[0,T]$
we use the classical notations $L^p(\Omega, \f_t, \Pi; U)$ and $L^p(0,T; U)$ for the classes of
$U$-valued $p$-Bochner-integrable functions on $\Omega$ and $(0,T)$, respectively
(without specifying the $\sigma$-algebra and the probability measure if $t=T$).
The symbol $C^0_w([0,T]; U)$ denotes the space of continuous functions from $[0,T]$ to
the space $U$ endowed with the weak topology. Moreover,
if $U_1$ and $U_2$ are separable Hilbert Spaces, we may write $\cL(U_1, U_2)$ and $\cL_2(U_1, U_2)$ to indicate 
the spaces of the linear continuous operators and Hilbert-Schmidt operators from $U_1$ to $U_2$, respectively.

Let $\tau$ be the trace operator $\tau: H^1(D)\rarr L^2(\Gamma)$.
Recall that the rank of $\tau$ coincides with the boundary Sobolev space $H^{1/2}(\Gamma)$
and there is a constant $M>0$ such that $\l|\tau u\r|_{H^{1/2}(\Gamma)}\leq M\l|u\r|_{H^1(D)}$ for any $u\in H^1(D)$
(see \cite[Thm.~7.39]{adams} and \cite[p.~14]{barbu_monot}):
hence the operator $\tau: H^1(D)\rarr H^{1/2}(\Gamma)$ is well-defined, linear and continuous.
Moreover, it is worth recalling that for any $k\geq1$, we have that $\tau \in \cL(H^k(D), H^{k-1/2}(\Gamma))$.

The symbol $\Delta_\Gamma$ denotes the usual Laplace-Beltrami operator on $L^2(\Gamma)$, i.e.
\[
\Delta_\Gamma:L^2(\Gamma)\rarr L^2(\Gamma)\,,\qquad
\Delta_\Gamma v:=\operatorname{div}_\Gamma\nabla_\Gamma v\,, \quad v\in D(\Delta_\Gamma):=H^2(\Gamma)\,,
\]
where $\nabla_\Gamma:=(\partial_{\tau_1}, \ldots, \partial_{\tau_{N-1}})$ is the Riemannian gradient on $\Gamma$
and $\partial_{\tau_i}$ is the derivative along the $i$-th tangential direction $\tau_i$ for every $i\in\{1,\ldots,N-1\}$.
Let us recall that the operator $-\Delta_\Gamma$ is maximal monotone on $L^2(\Gamma)$; moreover, for any $\delta>0$
and $k\in\En$,
its resolvent $(I-\delta\Delta_\Gamma)^{-1}$ belongs to $\cL(H^{k-1}(\Gamma), H^{k+1}(\Gamma))$.

For every $a,b\geq0$, we use the classical notation $a\lesssim b$ to mean that there
exists a positive constant $C$ such that $a\leq Cb$.

\subsection{Assumptions}
We precise here the assumptions that are in order throughout the work.

\noindent{\bf Assumptions on the double-well potentials.} We assume
\begin{gather*}
  \beta,\,\beta_\Gamma:\Ar\rarr2^\Ar \quad\text{maximal monotone}\,, \quad
  0\in\beta(0)\cap\beta_\Gamma(0)\,,\quad D(\beta)=D(\beta_\Gamma)=\Ar\,,\\
  \pi,\,\pi_\Gamma:\Ar\rarr\Ar \quad\text{Lipschitz continuous with Lipschitz constants } C_\pi, C_{\pi_\Gamma}\,. 
\end{gather*}
In this setting, the following proper, convex and lower semicontinuous functions are well-defined:
\[
  j,\,j_\Gamma:\Ar\rarr[0,+\infty)\,, \text{ such that }\quad\partial j=\beta\,, \quad \partial j_\Gamma=\beta_\Gamma\,, \quad j(0)=j_\Gamma(0)=0\,. 
\]
\carlo{Since $\beta$ and $\beta_\Gamma$ are everywhere defined, $j$ and $j_{\Gamma}$ are actually continuous. Moreover,} 
the convex conjugates of $j$ and $j_\Gamma$, i.e.
  \[
  j^*, j_\Gamma^*:\Ar\rarr[0,+\infty]\,, \qquad
  j^*(s):=\sup_{r\in\Ar}\{rs-j(r)\}\,, \quad
  j_\Gamma^*(s):=\sup_{r\in\Ar}\{rs-j_\Gamma(r)\}\,,
  \]
are superlinear at infinity (see \cite[Prop.~1.8]{barbu_monot}). Precisely, we have
  \[
  \lim_{|s|\rarr\infty}\frac{j^*(s)}{|s|}=\lim_{|s|\rarr\infty}\frac{j_\Gamma^*(s)}{|s|}=+\infty\,.
  \]
We need to make some further hypotheses on $j$ and $j_\Gamma$. Firstly, we require 
a symmetry property for the growth of the two potentials at infinity, namely
\[
  \limsup_{|r|\rarr\infty}\frac{j(r)}{j(-r)}<+\infty \qquad\text{and}\qquad
 \limsup_{|r|\rarr\infty}\frac{j_\Gamma(r)}{j_\Gamma(-r)}<+\infty\,,
\]
which is very common in literature (see \cite{mar-scar, mar-scar2, mar-scar3, scar, barbu, barbu-prato-rock}). 
Secondly, a natural assumption to make is that
\beq
  \tag{H1}\label{H1}
  j(r)\lesssim1+j_\Gamma(r)\,, \quad j_\Gamma(r)\lesssim1+j(r) \qquad\forall\,r\in\Ar\,,
\eeq
which means essentially that $j$ and $j_\Gamma$ control each other at $+\infty$. 
If we keep in mind the physical interpretation of the problem, \eqref{H1} is very reasonable
and can be reinterpreted as the requirement that the nonlinear flux on the boundary
is of the same type as the one in the interior of the domain.

\noindent However, note that condition \eqref{H1}
is much stronger than the corresponding one in the deterministic case,
in which it is sufficient to assume just one of the two inequalities (see \cite{cal-colli}).
Consequently, for sake of completeness, it is worth introducing two other possible hypotheses, 
in which the potentials are allowed to have different growth at infinity,
provided that they are bounded by specific polynomial functions:
\begin{gather}
  \tag{H2}\label{H2}
  j(r)\lesssim1+j_\Gamma(r)\,, \quad
  \begin{cases}
    j_\Gamma(r)\lesssim1+|r|^{\frac{2N}{N-2}} \quad&\text{if } N>2\\
    \exists\,p\geq1\,:\; j_\Gamma(r)\lesssim1+|r|^p \quad&\text{if } N=2
  \end{cases}
  \qquad\forall\,r\in\Ar\,,\\
   \tag{H3$_{\eps>0}$}\label{H3}
  j_\Gamma(r)\lesssim1+j(r)\,, \quad
  \begin{cases}
    j(r)\lesssim1+|r|^{\frac{2(N-1)}{N-3}} \quad&\text{if } N>3\\
    \exists\,p\geq1\,:\; j(r)\lesssim1+|r|^p \quad&\text{if } N=3\\
    \text{no restrictions on } j \quad&\text{if } N=2
  \end{cases}
  \qquad\forall\,r\in\Ar\,,\\
  \tag{H3$_{\eps=0}$}\label{H3'}
  j_\Gamma(r)\lesssim1+j(r)\,, \quad
  \begin{cases}
    j(r)\lesssim1+|r|^{\frac{2(N-1)}{N-2}} \quad&\text{if } N>2\\
    \exists\,p\geq1\,:\; j(r)\lesssim1+|r|^p \quad&\text{if } N=2
  \end{cases}
  \qquad\forall\,r\in\Ar\,.
\end{gather}
Let us comment on these conditions, focusing in particular on the cases
$N=2,3$, which are the most interesting in terms of applications.
In \eqref{H2} we are requiring that $j$ is controlled 
by $j_\Gamma$ and that $j_\Gamma$ is bounded by a polynomial of
degree six if $N=3$, or by any generic polynomial if $N=2$.
The second hypothesis requires instead that $j_\Gamma$ is controlled by $j$, and it
depends on wether we are working with $\eps>0$ or $\eps=0$.
In the case $\eps>0$, we can assume that $j$
has any polynomial growth if $N=3$, or any arbitrary growth if $N=2$.
In the case $\eps=0$, $j$ has to be controlled by a polynomial of degree four if $N=3$,
or by any generic polynomial if $N=2$.
In particular, note that the classical double-well potentials of degree four are included in
the interesting cases $N=2,3$.

\noindent Polynomial growths of this type for the potentials
have been widely used in the deterministic setting, also in the framework of
Cahn-Hilliard and quasilinear equations. Among the great literature, 
we can mention the works \cite{grass, grass2, escher, sprek-wu, cfp} and 
the references therein.

\noindent Throughout the paper, we will assume either hypothesis \eqref{H1} or \eqref{H2} or \eqref{H3}--\eqref{H3'}.

\noindent We introduce also the multivalued operator
\[
  \gamma:\Ar^2\rarr2^{\Ar^2}\,, \qquad\gamma(u,v):=\{(r,w)\in\Ar^2: r\in\beta(u), w\in\beta_\Gamma(v)\}\,, \quad (u,v)\in\Ar^2
\]
and the proper, convex and lower semicontinuous function
\[
  k:\Ar^2\rarr[0,+\infty)\,, \quad k(u,v):=j(u)+j_\Gamma(v)\,, \quad(u,v)\in\Ar^2\,:
\]
it is not difficult to check that $\gamma$ is maximal monotone on $\Ar^2$ and 
\[
  0\in\gamma(0)\,,\quad D(\gamma)=\Ar^2\,, \quad \gamma=\partial k\,, \quad k(0)=0\,.
\]
Finally, we define
\[
  \P:\Ar^2\rarr\Ar^2\,, \qquad 
  \P(u,v):=(\pi(u),\pi_\Gamma(v))\,, \quad(u,v)\in\Ar^2\,,
\]
which is Lipschitz continuous with Lipschitz constant $C_\P:=\max\{C_\pi, C_{\pi_\Gamma}\}$.

\noindent{\bf Assumptions on the variational setting.} For any $\eps\geq0$, we define the spaces
\[
  H:=L^2(D)\,, \quad H_\Gamma:=L^2(\Gamma)\,, \quad V:=H^1(D)\,, \quad
  V_\Gamma^\eps:=\begin{cases}
  H^1(\Gamma) \quad&\text{if}\quad \eps>0\,,\\
  H^{1/2}(\Gamma) &\text{if}\quad \eps=0\,,
  \end{cases}
\]
endowed with their natural norms $\l|\cdot\r|_H$, $\l|\cdot\r|_{H_\Gamma}$, 
$\l|\cdot\r|_V$ and $\l|\cdot\r|_{V^\eps_\Gamma}$, respectively.
We also use the notations $(\cdot,\cdot)_H$, $(\cdot,\cdot)_{H_\Gamma}$,
$\left<\cdot, \cdot\right>_V$ and $\left<\cdot, \cdot\right>_{V_\Gamma^\eps}$
for the standard scalar products of $H$ and $H_\Gamma$
and the duality pairings between $V^*$ and $V$, $(V_\Gamma^\eps)^*$ and $V_\Gamma^\eps$, respectively. 
In this way, for every $\eps\geq0$, $(V, H, V^*)$ 
and $(V_\Gamma^\eps, H_\Gamma, (V_\Gamma^\eps)^*)$ are Hilbert triplets with
compact inclusions $V\cembed H$ and $V_\Gamma^\eps\cembed H_\Gamma$.
The variational setting is obtained considering the product spaces
\[
  \H:=H\times H_\Gamma\,, \quad \V_\eps:=\{(u,v)\in V\times V^\eps_\Gamma: v=\tau u\}
\]
endowed with the norms
\begin{align*}
  \l|(u,v)\r|_\H&:=\sqrt{\l|u\r|^2_H+\l|v\r|^2_{H_\Gamma}} \qquad\forall\,(u,v)\in\H\,,\\
  \l|(u,v)\r|_{\V_\eps}&:=\sqrt{\l|u\r|^2_V+\l|v\r|^2_{V^\eps_\Gamma}} \qquad\forall\,(u,v)\in\V_\eps\,.
\end{align*}
Let us recall that, using the continuity of $\tau:H^1(D)\rarr H^{1/2}(\Gamma)$,
an equivalent norm on $\V_\eps$ is given by
\[
|(u,v)|_{\V_\eps}:=\sqrt{\l|u\r|_V^2+ \eps\l|\nabla v\r|_{H_\Gamma}^2} \qquad\forall(u,v)\in\V_\eps\,.
\]
\carlo{Notice that the above equivalence is not uniform in $\varepsilon$.}
It is clear that $\H$ is a Hilbert space with respect to the scalar product
\[
  \left((u_1,v_1),(u_2,v_2)\right)_\H:=(u_1,u_2)_H+(v_1, v_2)_{H_\Gamma} \quad\forall\,(u_1,v_1),(u_2,v_2)\in \H\,,
\]
and that, for every $\eps\geq0$,
$\V_\eps$ is included in $\H$ continuously and densely, so that $(\V_\eps, \H, \V_\eps^*)$
is a Hilbert triplet.
Now, it is natural to introduce the operator
\begin{gather*}
  \widetilde{\A}_\eps:C^\infty(\overline{D})\times 
  \tau\left(C^\infty(\overline{D})\right)\rarr C^\infty(\overline{D})\times C^\infty(\Gamma)\,,\\
  \widetilde{\A}_\eps(u,v):=\left(-\Delta u + \pi(u), \partial_{\bf n}u-\eps\Delta_\Gamma v + \pi_\Gamma(v)\right)\,.
\end{gather*}
Then it is immediate to check that for every $(\varphi, \psi)\in \V_\eps$ we have
\[
  \left(\widetilde{\A}_\eps(u,v), (\varphi, \psi)\right)_{\H}=\int_D\nabla u\cdot\nabla\varphi + \int_D\pi(u)\varphi
  +\eps\int_\Gamma\nabla_\Gamma v\cdot\nabla_\Gamma\psi + \int_\Gamma\pi_\Gamma(v)\psi\,.
\]
Analysing separately the cases $\eps>0$ and $\eps=0$ and
using the fact that $\tau: H^1(D)\rarr L^2(\Gamma)$ is continuous,
it is not difficult to check that the previous expression defines a linear continuous functional on $\V_\eps$ for every $\eps\geq0$.
Hence, the operator $\widetilde{\A}_\eps$ can be extended to 
\begin{gather*}
  \A_\eps: \V_\eps\rarr\V_\eps^*\,, \\
  \begin{split}
  \left<\A_\eps(u,v), (\varphi,\psi)\right>_{\V_\eps}:&=
  \int_D\nabla u\cdot\nabla\varphi + \int_D\pi(u)\varphi
  +\eps\int_\Gamma\nabla_\Gamma v\cdot\nabla_\Gamma\psi + \int_\Gamma\pi_\Gamma(v)\psi\\
  &\forall\,(u,v), (\varphi,\psi)\in\V_\eps\,.
  \end{split}
\end{gather*}
In the sequel, we will denote by $\C_\eps:\V_\eps\rarr\V_\eps^*$ the linear component of $\A_\eps$, i.e.
\carlo{\[
  \left<\C_\eps(u,v),(\varphi,\psi)\right>_{\V_\eps}:=\int_D\nabla u\cdot\nabla \varphi + 
  \eps\int_\Gamma\nabla_\Gamma v\cdot\nabla_\Gamma \psi \,, \quad
  (u,v), (\varphi,\psi) \in\V_\eps\,,
\]}
so that we have the representation $\A_\eps=\C_\eps+\P$, where we have used the same symbol for $\P$
and its corresponding Lipschitz operator induced on $\H$.

\noindent{\bf Assumptions on the noises. } Let $W$ and $W_\Gamma$ be two independent
cylindrical Wiener processes on two separable Hilbert spaces $U$ and $U_\Gamma$, respectively.
We introduce
\begin{gather*}
  B:\Omega\times[0,T]\times H\rarr\cL_2(U, H) \quad\text{progressively measurable}\,,\\
  B_\Gamma:\Omega\times[0,T]\times H_\Gamma\rarr\cL_2(U_\Gamma, H_\Gamma) \quad\text{progressively measurable}\,.
\end{gather*}
Then, setting
\begin{gather*}
  \mathcal{U}:=U\times U_\Gamma\,, \quad \mathcal{W}:=\begin{pmatrix}W\\W_\Gamma\end{pmatrix}\,,\\
  \B:\Omega\times[0,T]\times\H\rarr\cL_2(\U, \H)\,, \quad
  \B:=\begin{bmatrix}B & 0\\0 & B_\Gamma \end{bmatrix}\,,
\end{gather*}
we have that $\W$ is a cylindrical Wiener process on $\U$ and $\B$ is progressively measurable.
Moreover, we assume that $B$ and $B_\Gamma$ are
Lipschitz-continuous and at most with linear growth in their third arguments, uniformly on $\Omega\times[0,T]$, 
i.e.~that
there exists a positive constant $C$ such that
\begin{gather*}
  \l|B(\cdot,\cdot,u_1)-B(\cdot,\cdot,u_2)\r|_{\cL_2(U,H)}\leq C\l|u_1-u_2\r|_H \quad\forall\,u_1,u_2\in H\,,\\
  \l|B_\Gamma(\cdot,\cdot,v_1)-B_\Gamma(\cdot,\cdot,v_2)\r|_{\cL_2(U_\Gamma,H_\Gamma)}
  \leq C\l|v_1-v_2\r|_{H_\Gamma} \quad\forall\,v_1,v_2\in H_\Gamma\,,\\
  \l|B(\cdot,\cdot,u)\r|_{\cL_2(U,H)}\leq C\left(1+\l|u\r|_H\right) \quad\forall\, u\in H\,,\\
  \l|B(\cdot,\cdot,v)\r|_{\cL_2(U_\Gamma,H_\Gamma)}\leq C\left(1+\l|v\r|_{H_\Gamma}\right) \quad\forall\, u\in H_\Gamma\,.
\end{gather*}
Then, it is clear that the same hypotheses hold also for $\B$ in its corresponding spaces.

\noindent{\bf Assumption on the initial datum.} We assume that the initial datum satisfies
\[
  (x_0,y_0) \in L^2\left(\Omega, \f_0, \Pi; \H\right)\,.
\]

\subsection{Formulation of the problem and main results}
In this setting, we can write the SPDE of the joint process $(x_t, y_t)$ as follows:
\beq
  \label{prob}
  d(x_t,y_t) + \A_\eps(x_t, y_t)\,dt + \gamma(x_t, y_t)\,dt \ni \B(t,x_t, y_t)\,d\W_t\,, \quad (x(0), y(0))=(x_0, y_0)\,.
\eeq

\begin{defin}
  \label{solution}
  A strong solution to problem \eqref{prob} is a quadruplet $(x,y,\xi,\xi_\Gamma)$ such that
  \begin{gather}
  (x,y) \in L^2\left(\Omega; L^\infty(0,T; \H)\right) \cap L^2\left(\Omega\times(0,T); \V_\eps\right)\,,\\
  \xi \in L^1\left(\Omega\times(0,T)\times D\right)\,, \qquad \xi_\Gamma \in L^1\left(\Omega\times(0,T)\times \Gamma\right)\,,\\
  (x,y) \in C^0_w\left([0,T]; \H\right) \quad\Pi\text{-a.s.}\,,\\
  j(x)+j^*(\xi)\in L^1\left(\Omega\times(0,T)\times D\right)\,, \qquad
  j_\Gamma(y)+j^*_\Gamma(\xi_\Gamma)\in L^1\left(\Omega\times(0,T)\times \Gamma\right)\,,\\
  \xi\in\beta(x) \quad\text{a.e.~in } \Omega\times(0,T)\times D\,, \qquad
  \xi_\Gamma\in\beta_\Gamma(y) \quad\text{a.e.~in } \Omega\times(0,T)\times \Gamma\,,\\
  (x,y,\xi,\xi_\Gamma) \quad\text{is predictable}\,,\\
  \B(\cdot,x,y) \quad\text{is progressively measurable}\,,\\
  \begin{split}
  (x(t), y(t)) &+ \int_0^t\A_\eps(x(s), y(s))\,ds + \int_0^t(\xi(s), \xi_\Gamma(s))\,ds\\
  &=(x_0, y_0) + \int_0^t\B(s,x(s),y(s))\,d\W_s \quad\forall\,t\in[0,T]\,, \quad\Pi\text{-a.s.}\,.
  \end{split}
  \end{gather}
\end{defin}

\begin{defin}
  \label{well-p}
  We say that problem \eqref{prob} is well-posed for a given $\eps\geq0$ if
  for any initial datum $(x_0,y_0)\in L^2(\Omega,\f_0,\Pi; \H)$
  there exists a unique strong solution to \eqref{prob} in the sense of Definition~\ref{solution} and the
  following solution map is Lipschitz-continuous:
  \[
  \Lambda_\eps: L^2\left(\Omega; \H\right) \rarr L^2\left(\Omega; L^\infty(0,T; \H)\right) \cap 
  L^2\left(\Omega\times(0,T); \V_\eps\right)\,,\qquad
  (x_0,y_0) \mapsto (x,y)\,.
  \]
\end{defin}

\begin{thm}
  \label{teo1}
  The problem \eqref{prob} is well-posed for any $\eps\geq0$.
\end{thm}

\begin{thm}
  \label{teo2}
  Let $(x_\eps, y_\eps, \xi_\eps, \xi_{\Gamma,\eps})$ and $(x,y,\xi,\xi_\Gamma)$ be the unique strong solutions to \eqref{prob}
  with additive noise
  given by Theorem \ref{teo1} in the cases $\eps>0$ and $\eps=0$, respectively. 
 Then for every sequence $\{\eps_k\}_{k\in\En}$, with $\eps_k\searrow0$ as $k\rarr\infty$,
  we have, as $k\rarr\infty$,
  \begin{gather}
  \label{c1}
  (x_{\eps_k}, y_{\eps_k})\rarr (x,y) \quad\text{in } L^2\left(0,T; \H\right)\,,\quad\Pi\text{-a.s.}\,,\\
  \label{c2}
  (x_{\eps_k}, y_{\eps_k})\weakstar (x,y) \quad\text{in } L^\infty\left(0,T; L^2(\Omega; \H)\right)\,,\\
  \label{c3}
  (x_{\eps_k}, y_{\eps_k})\rarrw (x,y) \quad\text{in } L^2\left(\Omega\times(0,T); \V_0\right)\,,\\
  \label{c4}
  \xi_{\eps_k}\rarrw\xi \quad\text{in } L^1\left(\Omega\times(0,T)\times D\right)\,, \qquad
  \xi_{\Gamma,\eps_k}\rarrw\xi_{\Gamma} \quad\text{in } L^1\left(\Omega\times(0,T)\times\Gamma\right)\,,\\
  \label{c5}
  \eps_k y_{\eps_k} \rarr0 \quad\text{in } L^2\left(\Omega\times(0,T); H^1(\Gamma)\right)\,.
  \end{gather}
\end{thm}


\section{Well-posedness}
\setcounter{equation}{0}
\label{proof-well}

First of all, we prove well-posedness for the problem \eqref{prob} with additive and more regular noise.
Namely, let $\Z$ be a separable Hilbert space such that $\Z\embed \V_\eps\cap L^\infty(D)\times L^\infty(\Gamma)$
(such a $\Z$ exists thanks to the Sobolev embeddings theorems) and consider the problem
\beq
  \label{additive}
  d(x_t,y_t) + \A_\eps(x_t, y_t)\,dt + \gamma(x_t, y_t)\,dt \ni \B(t)\,d\W_t\,, \quad (x(0), y(0))=(x_0, y_0)\,,
\eeq
where
\beq
\label{B_strong}
\B\in L^2\left(\Omega; L^2(0,T; \cL_2(\U,\Z))\right)\,.
\eeq
The hypothesis on $\B$ will be removed at the end of the section.
Existence of a solution is proved using a suitable approximation on the equation and then
passing to the limit using compactness results and monotonicity arguments.
Continuous dependence on the data is obtained using an appropriate version of 
It\^o's formula.

Throughout the section, $\eps\geq0$ is a fixed constant, so that the argument fits both to the case $\eps=0$ and $\eps>0$
at the same time;
when two different approaches are needed, we will specify it explicitly.

\subsection{The approximated problem}
For any $\lambda\in(0,1)$, let $\beta_\lambda$, $\beta_{\Gamma, \lambda}$, $j_\lambda$ and $j_{\Gamma, \lambda}$
denote the Yosida approximations of the graphs $\beta$ and $\beta_\Gamma$
and the Moreau regularizations of the functions $j$ and $j_{\Gamma}$, respectively.
With this notation, it is a standard matter to check that the
Yosida approximation of $\gamma$ and the Moreau regularization of $k$ are given by
 \[\gamma_\lambda=(\beta_\lambda, \beta_{\Gamma, \lambda})\,, \qquad
 k_\lambda(u,v)=j_\lambda(u)+j_{\Gamma, \lambda}(v) \quad\forall\,(u,v)\in\Ar^2\,.
\]

We consider the following approximated problem:
\begin{gather}
\label{app}
  d(x_\lambda,y_\lambda) + \A_\eps(x_\lambda, y_\lambda)\,dt +
  \gamma_\lambda(x_\lambda, y_\lambda)\,dt = \B(t)\,d\W_t\,,\\
  \label{app_init}
  (x_\lambda(0), y_\lambda(0))=(x_0, y_0)\,.
\end{gather}
For sake of simplicity, let us use the notation
\[
  \A_{\eps,\lambda}:\V_\eps\rarr\V_\eps^*\,, \qquad
  \A_{\eps,\lambda}(u,v):=\A_\eps(u,v)+\gamma_\lambda(u,v) \quad\forall\,(u,v)\in\V_\eps\,,
\]
so that we can write the approximated problem as
\[
  d(x_\lambda,y_\lambda) + \A_{\eps,\lambda}(x_\lambda,y_\lambda)\,dt = \B(t)\,d\W_t\,, \qquad
  (x_\lambda(0),y_\lambda(0))=(x_0, y_0)\,.
\]
We recall some properties of the operator $\A_{\eps,\lambda}$ in the following lemma.

\begin{lem}
  \label{properties}
  The operator $\A_{\eps,\lambda}:\V_\eps\rarr\V_\eps^*$ is hemicontinuous,
  weakly monotone, weakly coercive and bounded. More specifically, there exist two positive constants 
  $C$ and $C_{\eps,\lambda}$,
  with the first being independent of $\eps$ and $\lambda$,
  such that the following conditions hold for any $(u_1,v_1), (u_2, v_2), (u,v) \in \V_\eps$:
  \begin{gather*}
    \left<\A_{\eps,\lambda}(u_1,v_1)-\A_{\eps,\lambda}(u_2,v_2), (u_1,v_1)-(u_2,v_2)\right>_{\V_\eps}
    \geq-C\l|(u_1,v_1)-(u_2,v_2)\r|^2_\H\,,\\
    \left<\A_{\eps,\lambda}(u,v), (u,v)\right>_{\V_\eps}\geq\l|\nabla u\r|^2_H+\eps\l|\nabla_\Gamma v\r|_{H_\Gamma}^2
    -C\l|(u,v)\r|_\H^2-C\,,\\
    \l|\A_{\eps,\lambda}(u,v)\r|_{\V_\eps^*} \leq C_{\eps,\lambda}\left(1+\l|(u,v)\r|_{\V_\eps}\right)\,.
  \end{gather*}
\end{lem}
\begin{proof}
  Firstly, let $(r_1, s_1), (r_2, s_2), (r_3,s_3)\in\V_\eps$: for any $t\in\Ar$ we have
  \[
  \begin{split}
  &\left<\A_{\eps,\lambda}((r_1,s_1)+t(r_2,s_2)), (r_3,s_3)\right>_{\V_\eps}\\
  &\qquad=\int_D\nabla(r_1+tr_2)\cdot\nabla r_3 + \int_D\pi(r_1+tr_2)r_3
  + \int_D\beta_\lambda(r_1+tr_2)r_3\\
  &\qquad+\eps\int_\Gamma\nabla_\Gamma(s_1+ts_2)\cdot\nabla_\Gamma s_3
  +\int_\Gamma\pi_\Gamma(s_1+ts_2)s_3 + \int_\Gamma\beta_{\Gamma, \lambda}(s_1+ts_2)s_3\,,
  \end{split}
  \] 
  and by the Lipschitz continuity of $\pi$, $\pi_\Gamma$, $\beta_\lambda$ and $\beta_{\Gamma, \lambda}$,
  the right-hand side is a continuous function of $t$. Hence, $\A_{\eps,\lambda}$ is hemicontinuous.
  Secondly, using the Lipschitz continuity of $\pi$ and $\pi_\Gamma$ and the monotonicity of
  $\beta_\lambda$ and $\beta_{\Gamma, \lambda}$, we have
  \[
  \begin{split}
  &\left<\A_{\eps,\lambda}(u_1,v_1)-\A_{\eps,\lambda}(u_2,v_2), (u_1,v_1)-(u_2,v_2)\right>_{\V_\eps}
  =\int_D|\nabla(u_1-u_2)|^2 \\
  &+ \int_D\left(\pi(u_1)-\pi(u_2)\right)(u_1-u_2) +
  \int_D\left(\beta_\lambda(u_1)-\beta_\lambda(u_2)\right)(u_1-u_2)
  +\eps\int_\Gamma|\nabla_\Gamma(v_1-v_2)|^2\\ 
  &+ \int_\Gamma\left(\pi_\Gamma(v_1)-\pi_\Gamma(v_2)\right)(v_1-v_2) + 
  \int_\Gamma\left(\beta_{\Gamma,\lambda}(v_1)-\beta_{\Gamma, \lambda}(v_2)\right)(v_1-v_2)\\
  &\geq-C_\pi\l|u_1-u_2\r|_H^2
  -C_{\pi_\Gamma}\l|v_1-v_2\r|_{H_\Gamma}^2
  \geq -C_\P\l|(u_1,v_1)-(u_2,v_2)\r|^2_\H\,,
  \end{split}
  \]
  from which the weak monotonicity. Moreover, using the Lipschitz continuity of $\pi$ and $\pi_\Gamma$,
  a similar computation leads to 
  \[
  \begin{split}
  &\left<\A_\eps(u,v)+\gamma_\lambda(u,v), (u,v)\right>_{\V_\eps}\\
  &\qquad\qquad\geq \int_D|\nabla u|^2 + \int_D\pi(u)u + \int_D\beta_\lambda(u)u
  +\eps\int_\Gamma|\nabla_\Gamma v|^2 + \int_\Gamma \pi_\Gamma(v)v + \int_\Gamma\beta_{\Gamma, \lambda}(v)v\\
  &\qquad\qquad\geq \l|\nabla u\r|_H^2 + \eps\l|\nabla_\Gamma v\r|_{H_\Gamma}^2 - C\left(1+\l|(u,v)\r|_\H^2\right)
  \end{split}
  \]
  for a positive constant $C$,
  from which we deduce the weak coercivity. Indeed, this is immediate if $\eps>0$; if $\eps=0$, this
  follows from the fact that the norm $\l|\cdot\r|_{\V_0}$ is equivalent to $|\cdot|_{\V_0}$
  (since $\tau:V\rarr V_{\Gamma}^0$ is continuous).
  Finally, for any $(\varphi,\psi)\in \V_\eps$, by the Lipschitz continuity of 
  $\pi$, $\pi_\Gamma$, $\beta_\lambda$ and $\beta_{\Gamma,\lambda}$,
  using the H\"older inequality and renominating the positive constant $C$ at each passage we have
  \[
  \begin{split}
  &\left<\A_{\eps,\lambda}(u,v), (\varphi,\psi)\right>_{\V_\eps}=
  \int_D\nabla u\cdot\nabla\varphi + \int_D\pi(u)\varphi
  +\int_D\beta_\lambda(u)\varphi\\
  &\qquad\qquad\qquad\qquad\quad+\eps\int_\Gamma\nabla_\Gamma v\cdot \nabla_\Gamma\psi +
   \int_\Gamma\pi_\Gamma(v)\psi 
  + \int_\Gamma\beta_{\Gamma,\lambda}(v)\psi\\
  &\leq\l|\nabla u\r|_H\l|\nabla\varphi\r|_H + C(1+\l|u\r|_H)\l|\varphi\r|_H
  + \frac1\lambda\l|u\r|_H\l|\varphi\r|_H\\
  &\qquad\qquad\qquad\qquad+\eps\l|\nabla_\Gamma v\r|_{H_\Gamma}\l|\nabla_\Gamma\psi\r|_{H_\Gamma}
  +C(1+\l|v\r|_{H_\Gamma})\l|\psi\r|_{H_\Gamma} + \frac1\lambda\l|v\r|_{H_\Gamma}\l|\psi\r|_{H_\Gamma}\\
  &\leq C\left(\max\{1,\eps\}+1+\frac1\lambda\right)\left((1+\l|u\r|_V)\l|\varphi\r|_V 
  + (1+\l|v\r|_{V_\Gamma^\eps})\l|\psi\r|_{V_\Gamma^\eps}\right)\\
  &\leq C_{\eps,\lambda}\left(1+\l|(u,v)\r|_{\V_\eps}\right)\l|(\varphi,\psi)\r|_{\V_\eps}\,,
  \end{split}
  \]
  from which the boundedness follows. 
\end{proof}

The previous lemma ensures that the approximated problem is well-posed according to the classical variational approach 
by Pardoux, Krylov and Rozovski{\u\i} (see \cite{KR-spde, Pard, prevot-rock}) in the Gelfand triple $(\V_\eps, \H, \V_\eps^*)$.
Hence, for any $\lambda\in(0,1)$ there is a unique strong solution
\[
  (x_\lambda, y_\lambda) \in L^2\left(\Omega; C^0([0,T]; \H)\right)\cap L^2\left(\Omega\times(0,T); \V_\eps\right)
\]
to the approximated problem \eqref{app}--\eqref{app_init}. Moreover we can exhibit some a priori estimates, 
as it is shown in the following lemmata.

\begin{lem}\label{ito}
There exists a constant $K$ such that the following inequality holds
\begin{equation*}
\begin{split}
&\l|(x_\lambda,y_\lambda)\r|_{L^2\left(\Omega; C^0([0,T]; \H)\right)}^2 + 
\l|(x_\lambda,y_\lambda)\r|_{L^2\left(\Omega; L^2(0,T; \V_\eps)\right)}^2 + 
\l|\beta_\lambda(x_\lambda) x_\lambda\r|_{L^1\left(\Omega; L^1(0,T; L^1(D))\right)}\\
&+ \l|\beta_{\Gamma,\lambda}(y_\lambda) y_\lambda\r|_{L^1\left(\Omega; L^1(0,T; L^1(\Gamma))\right)} \leq 
K \left(1+ \l|(x_0,y_0)\r|^2_{L^2\left(\Omega;\H\right)}   + \l| \B \r|_{L^2\left(\Omega; L^2(0,T; \cL_2(\U,\H)\right))}^2 \right)\,.
\end{split}
\end{equation*}
\end{lem}
\begin{proof}
The proof relies on the application of the version of It\^o formula introduced in \cite{KR-spde}. Precisely,
we have for every $t\in[0,T]$ and $\Pi$-almost surely that
\begin{equation*}
\begin{split}
&\l|(x_\lambda(t),y_\lambda(t))\r|^2_\H + 
2 \int_0^t \left<  \A_{\eps,\lambda} (x_\lambda,y_\lambda), (x_\lambda,y_\lambda) \right>_{\V_\eps}\,ds\\
&\qquad= \l|(x_0,y_0)\r|^2_{L^2\left(\Omega;\H\right)} + 2\int_0^t (x_\lambda(s),y_\lambda(s)) \B(s)\,d\W_s + 
\int_0^t \l| \B(s) \r|^2_{\cL_2(\U,\H)}\,ds\,.
\end{split}
\end{equation*} 
Using the Lipschitzianity of $\pi$ and $\pi_\Gamma$ and the
weak coercivity of $\A_{\eps,\lambda}$, we deduce that
\begin{equation*}
\begin{split}
&\l|(x_\lambda(t),y_\lambda(t))\r|^2_\H + 2C \int_0^t \l|(x_\lambda(s),y_\lambda(s))\r|_{\V_\eps}^2\,ds  
- 2C \int_0^t \left( 1+ \l|(x_\lambda(s),y_\lambda(s))\r|^2_{\H}\right)\,ds\\
&\quad+2\int_0^t\int_D\beta_\lambda(x_\lambda(s))x_\lambda(s)\,ds 
+ 2\int_0^t\int_\Gamma\beta_{\Gamma,\lambda}(y_\lambda(s))y_\lambda(s)\,ds\\
&\quad \lesssim \l|(x_0,y_0)\r|^2_{L^2\left(\Omega;\H\right)} + 
2\int_0^t (x_\lambda(s),y_\lambda(s)) \B(s)\,d\W_s + \int_0^t \l| \B(s) \r|^2_{\cL_2(\U,\H)}\,ds
\end{split}
\end{equation*} 
for a positive constant $C$ independent of $\lambda$.
Taking the supremum in time and expectation, thanks to the Gronwall lemma we get
\begin{equation*}
\begin{split}
&\E\l|(x_\lambda(t),y_\lambda(t))\r|^2_{C([0,T];\H)} + \E\l|(x_\lambda,y_\lambda)\r|^2_{L^2(0,T; \V_\eps)}\\
&\quad+ \E \int_0^T\int_D \beta_\lambda(x_\lambda(s)) x_\lambda(s)\,ds
+\E \int_0^T \int_\Gamma\beta_{\Gamma, \lambda}(y_\lambda(s)) y_\lambda(s)\,ds \\
&\quad  \lesssim 1 + \E\l|(x_0,y_0)\r|^2_{L^2\left(\Omega;\H\right)} +
\E\sup_{t \in[0,T]} \left| \int_0^t (x_\lambda,y_\lambda) \B(s) d\W_s \right| + \E \l| \B(s) \r|^2_{L^2(0,T; \cL_2(\U,\H))}\,.
\end{split}
\end{equation*} 
A direct consequence of the Burkholder-Davis-Gundy and Young inequalities
(the reader can refer to \cite[Lem.~3.1]{mar-scar}) is that
for every $\delta>0$ we have
\begin{equation*}
\E\sup_{t \in[0,T]} \left| \int_0^t (x_\lambda,y_\lambda) \B(s)\,d\W_s \right| \leq 
\delta \E \l| (x_\lambda,y_\lambda) \r|^2_{C([0,T];\H)} + K(\delta) \E \int_0^T\l| \B(s) \r|^2_{\cL_2(\U,\H)}\,ds\,.
\end{equation*}
From the arbitrariness of $\delta$, we conclude choosing $\delta$ small enough.
\end{proof}

\begin{lem}\label{pathwise}
There exists $\Omega'\in\f$ with $\Pi(\Omega')=1$ such that, for every $\omega\in\Omega'$,
  there exists a positive constant \carlo{$K = K(\omega)$ satisfying
  \[\begin{split}
  \l|(x_\lambda,y_\lambda)(\omega)\r|^2_{L^\infty(0,T; \H)\cap L^2(0,T; \V_\eps)}
  &+\l|\beta_\lambda(x_\lambda(\omega))x_\lambda(\omega)\r|_{L^1((0,T)\times D)}\\
  &+\l|\beta_{\Gamma,\lambda}(y_\lambda(\omega))y_\lambda(\omega)\r|_{L^1((0,T)\times\Gamma)}
  \leq K \qquad\forall\,\lambda\in(0,1)\,.
  \end{split}\]}
\end{lem}
\begin{proof}
\carlo{To shorten the notation, we will use the notation $\B \cdot \W$ to mean the stochastic integral of $\B$ with respect to $\W$. }
The approximated equation can be written as
  \[
  \frac{d}{dt} \big( (x_\lambda,y_\lambda) - \B\cdot\W\big) + \A_\eps(x_\lambda,y_\lambda)
  + \gamma_\lambda(x_\lambda, y_\lambda) = 0 \quad\text{a.e.~in } (0,T)\,, \quad\Pi\text{-a.s.}
  \]
  Moreover, thanks to \eqref{B_strong} and the choice of $\Z$ we have
  \[
  \B\cdot\W \in L^2(0,T; \V_\eps)\cap L^\infty(0,T; L^\infty(D)\times L^\infty(\Gamma)) \quad\Pi\text{-a.s.}
  \]
  Let then $\Omega'\in\f$ with $\Pi(\Omega')=1$ such that the two previous relations hold and fix $\omega\in \Omega'$.
  Testing (deterministically) the first equation by $(x_\lambda,y_\lambda)-\B\cdot\W$, we get
  \[
  \begin{split}
  \frac12&\l|(x_\lambda,y_\lambda)(t)-\B\cdot\W(t)\r|_\H^2 + 
  \int_0^t\left<A_\eps(x_\lambda, y_\lambda)(s), (x_\lambda,y_\lambda)(s)-\B\cdot\W(s)\right>_{\V_\eps}\,ds\\
  &+\int_0^t\int_D\beta_\lambda(x_\lambda(s))(x_\lambda(s)-B\cdot W(s))\,ds + 
  \int_0^t\int_\Gamma\beta_{\Gamma,\lambda}(y_\lambda(s))(y_\lambda(s)-B_\Gamma\cdot W_\Gamma(s))\,ds\\
  &=\frac12\l|(x_0,y_0)\r|_\H^2 \qquad\forall\,t\in[0,T]\,.
  \end{split}
  \]
  Rearranging the terms, using the regularity of $\B\cdot\W$, the Young inequality
  and the Lipschitzianity of $\pi$ and $\pi_\Gamma$ we infer that
  \[\begin{split}
  &\frac12\l|(x_\lambda,y_\lambda)(t)-\B\cdot\W(t)\r|_\H^2 + \int_0^t\l|\nabla x_\lambda(s)\r|_{H}^2\,ds
  +\eps\int_0^t\l|\nabla_\Gamma y_\lambda(s)\r|_{H_\Gamma}^2\,ds\\
  &\qquad\quad+\int_0^t\int_D\beta_\lambda(x_\lambda(s))x_\lambda(s)\,ds + 
  \int_0^t\int_\Gamma\beta_{\Gamma,\lambda}(y_\lambda(s))y_\lambda(s)\,ds\\
  &=\frac12\l|(x_0,y_0)\r|_\H^2 + \int_0^t\int_D\nabla x_\lambda(s)\cdot\nabla (B\cdot W)(s)\,ds
  +\eps\int_0^t\int_\Gamma\nabla_\Gamma y_\lambda(s)\cdot\nabla_\Gamma (B_\Gamma\cdot W_\Gamma)(s)\,ds\\
  &\qquad\quad+\int_0^t\int_D\beta_\lambda(x_\lambda(s))B\cdot W(s)\,ds + 
  \int_0^t\int_\Gamma \beta_{\Gamma,\lambda}(y_\lambda(s))B_\Gamma\cdot W_\Gamma(s)\,ds\\
  &\qquad\quad-\int_0^t\int_D\pi(x_\lambda(s))(x_\lambda(s)-B\cdot W(s))\,ds
  -\int_0^t\int_\Gamma\pi_\Gamma(y_\lambda(s))(y_\lambda(s)-B_\Gamma\cdot W_\Gamma(s))\,ds\\
  &\leq\frac12\l|(x_0,y_0)\r|_\H^2 + \frac12 \int_0^t\l|\nabla x_\lambda(s)\r|_{H}^2\,ds
  +\frac\eps2\int_0^t\l|\nabla_\Gamma y_\lambda(s)\r|_{H_\Gamma}^2\,ds\\
  &\qquad\quad+\frac{1\vee \eps}{2}\l|\B\cdot \W\r|^2_{L^2(0,T; \V_\eps)}
  +\frac12\int_0^t\int_Dj^*(\beta_\lambda(x_\lambda(s)))\,ds
  +\frac12\int_0^t\int_\Gamma j_\Gamma^*(\beta_{\Gamma, \lambda}(y_\lambda(s)))\,ds\\
  &\qquad\quad+\int_0^T\int_Dj(2(B\cdot W))
  + \int_0^T\int_\Gamma j_\Gamma(2(B_\Gamma\cdot W_\Gamma))\\
  &\qquad\quad+\left(\C_\P+\frac12\right)\int_0^t\l|(x_\lambda,y_\lambda)(s)-\B\cdot\W(s)\r|_\H^2\,ds
  +\frac{\C_\P^2}2\l|\B\cdot\W\r|^2_{L^2(0,T; \H)}\,.
  \end{split}\]
  Since $B\cdot W \in L^\infty((0,T)\times D)$ and $B_\Gamma\cdot W_\Gamma \in L^\infty((0,T)\times\Gamma)$,
  by continuity of $j$ and $j_\Gamma$ we have that 
  $j(2(B\cdot W))\in L^1((0,T)\times D)$ and $j_\Gamma(2(B_\Gamma\cdot W_\Gamma))\in L^1((0,T)\times \Gamma)$.
  Hence, recalling that on the left-hand side
  \[
  \beta_\lambda(x_\lambda)x_\lambda=j((I+\lambda\beta)^{-1}x_\lambda) + j^*(\beta_\lambda(x_\lambda))\,, \quad
  \beta_{\Gamma,\lambda}(y_\lambda)y_\lambda=j((I+\lambda\beta_\Gamma)^{-1}y_\lambda) + 
  j^*(\beta_{\Gamma,\lambda}(y_\lambda))\,,
  \]
  rearranging the terms and using the Gronwall lemma we can conclude.
\end{proof}
 
\begin{prop}\label{convergence}
For any $\omega \in \Omega'$ we can extract a 
subsequence $\lambda' = \lambda'(\omega)$ of $\lambda$ for which the following convergences hold as $\lambda' \to 0$:
\begin{align*}
&(x_{\lambda'},y_{\lambda'})(\omega,\cdot) \weakstar (x,y)(\omega,\cdot) &&\text{ in } L^\infty(0,T;\H)\,,\\
&(x_{\lambda'},y_{\lambda'})(\omega,\cdot) \rightharpoonup (x,y)(\omega,\cdot) &&\text{ in } L^2(0,T;\V_\eps)\,,\\
&\beta_{\lambda'}(x_{\lambda'}(\omega,\cdot)) \rightharpoonup \xi(\omega,\cdot) &&\text{ in } L^1((0,T)\times D)\,,\\
&\beta_{\Gamma,\lambda'}(y_{\lambda'}(\omega,\cdot)) \rightharpoonup \xi_\Gamma(\omega,\cdot) 
&&\text{ in } L^1((0,T)\times \Gamma)\,,\\
&(x_{\lambda'},y_{\lambda'})(\omega,\cdot) \to (x,y)(\omega,\cdot) &&\text{ in } L^2(0,T;\H)\,.
\end{align*}
\end{prop}

\begin{proof}
Let us fix $\omega\in\Omega'$. The first two convergences follow from Lemma~\ref{pathwise}.
Regarding the third one, recall that for any $u \in \Ar$ we have  
$\beta_\lambda(u) \in \partial j \left( (I + \lambda\beta)^{-1} u\right) = 
\beta\left( (I+\lambda\beta)^{-1}u\right)$ and $\beta_\lambda(u)(I+\lambda\beta)^{-1}u \geq 0$, 
so that
\begin{equation*}
j\left( (I+ \lambda\beta)^{-1}u \right) + j^*\left( \beta_\lambda(u)\right) = \beta_\lambda(u)(I+ \lambda\beta)^{-1}u 
\leq \beta_\lambda(u)u \qquad \forall \, u \in \Ar
\end{equation*}
thanks to the contraction property of the resolvent operator. 
Thanks to Lemma~\ref{pathwise} and
since $j^*$ is superlinear at infinity, for any $\omega\in\Omega'$
the sequence $(\beta_\lambda(x_\lambda(\omega)))_\lambda$ turns out 
to be weakly relatively compact  in $L^1((0,T) \times D)$ thanks to the de la Vall\'ee Poussin criterion 
along with Dunford-Pettis theorem.
Hence, we can extract a subsequence $\lambda'(\omega)$ 
which satisfies the required convergence. 
The same reasoning can be applied to $\beta_{\Gamma,\lambda}$ to get the fourth convergence statement. 
Finally, it remains to show the strong convergence of $(x_{\lambda'},y_{\lambda'})$ in $L^2(0,T;\H)$. 
To this end, going back to
\begin{equation*}
\frac{d}{dt} \big( (x_\lambda,y_\lambda) - \B\cdot\W\big) + \A_\eps(x_\lambda,y_\lambda)
  + \gamma_\lambda(x_\lambda, y_\lambda) = 0 \quad\text{a.e.~in } (0,T)\,,
\end{equation*}  
from Lemma~\ref{pathwise}, the boundedness of $\A_\eps$ and the fact that $\V_{\eps}^*\embed \mathcal{Z}^*$,
we have
\begin{equation*}
\begin{split}
&\l| \A_\eps(x_\lambda,y_\lambda)\r|_{L^1(0,T; \mathcal{Z}^*)} \lesssim
1+\l| (x_\lambda, y_\lambda) \r|_{L^2(0,T; \V_{\eps})}, \\
&\l| \gamma_\lambda(x_\lambda, y_\lambda)  \r|_{L^1(0,T; \mathcal{Z}^*)} \lesssim
\l| \gamma_\lambda(x_\lambda, y_\lambda)  \r|_{L^1(0,T; L^1(D) \times L^1(\Gamma))}.
\end{split}
\end{equation*}
Hence, by Lemma \ref{ito}, $\l|\frac{d}{dt}((x_\lambda,y_\lambda) - \B\cdot\W) \r|_{L^1(0,T; \mathcal{Z}^*)}$ 
is uniformly bounded in $\lambda$ and 
we can apply Simon's theorem (see \cite[Cor.~4, p.~85]{simon}) to get that 
$(x_\lambda, y_\lambda)$ is relatively compact in $L^2(0,T;\H)$. 
Then the weak convergence of  $(x_{\lambda'}, y_{\lambda'})$ towards $(x,y)$ in $L^2(0,T;\V_\varepsilon)$, implies that
\begin{equation*}
(x_{\lambda'}, y_{\lambda'}) \to (x,y) \qquad \text{ strongly in } L^2(0,T;\H),
\end{equation*} 
which is the required convergence.
\end{proof}

\subsection{The limit problem}
\label{limit_problem}

Now we are ready for the proof of the well-posedness of the equation \eqref{additive}, where the noise enter the system in an additive fashion. We divide the proof in several steps:

\textbf{Identification of the limit.}
Fix $\omega \in \Omega'$: in the sequel we do not emphasize the $\omega$-dependence as no confusion can arise. 
By Proposition \ref{convergence}, $(x_{\lambda'},y_{\lambda'}) \to (x,y)$ strongly in $L^2(0,T;\H)$, and
$(x_{\lambda'}(t),y_{\lambda'}(t)) \to (x(t),y(t))$ for a.e.~$t \in [0,T]$ up to passing to a further subsequence.  
As for the $\A_\eps$-part we write $\A_\eps = \C_\eps + \P$.
Firstly, we employ weak convergence $(x_{\lambda'},y_{\lambda'}) \rightharpoonup (x,y)$ in $L^2(0,T;\V_\eps)$ to get 
\begin{equation*}
\int_0^t \C_\eps(x_{\lambda'}(s),y_{\lambda'}(s))\,ds \rightharpoonup \int_0^t \C_\eps(x(s),y(s))\,ds \quad \text{ in } \V_\eps^*
\end{equation*}
for all $t \in [0,T]$. This is straightforward setting $\phi_0 \in \V_\eps$ and choosing as a test function 
$\phi: = s \mapsto 1_{[0,t]}(s)\phi_0 \in L^2(0,T;\V_\eps)$. 
Secondly, we simply use the strong convergence $(x_{\lambda'},y_{\lambda'}) \to (x,y)$ in $L^2(0,T;\H)$
and the Lipschitz continuity of $\P$
to get 
\begin{equation*}
\int_0^t\P(x_{\lambda'}(s),y_{\lambda'}(s))\,ds \to  \int_0^t \P(x(s), y(s))\,ds \quad\text{ in } \H
\end{equation*}
for every time $t \in [0,T]$. Regarding the monotone part we employ the weak convergence in 
$L^1(0,T; L^1(D) \times L^1(\Gamma))$ of $\gamma_{\lambda'}$  towards $(\xi,\xi_\Gamma)$ to easily get that
\begin{equation*}
\int_0^t \gamma_{\lambda'}(x_{\lambda'}(s), y_{\lambda'}(s))\,ds \rightharpoonup  
\int_0^t \big( \xi(s), \xi_{\Gamma}(s) \big)\,ds \quad\text{ in } L^1(D)\times L^1(\Gamma)
\end{equation*}
for all $t \in [0,T]$.
Summing up all the previous convergences, we get the limit equation
\begin{equation}\label{first limit equation}
\begin{split}
(x(t),y(t)) &+ \int_0^t \A_\eps(x(s),y(s))\,ds + \int_0^t (\xi(s),\xi_\Gamma(s))\,ds\\
&= (x_0, y_0) + \int_0^t \B(s)\,d\W_s, \quad \text{ in } \mathcal{Z}^*
\end{split}
\end{equation}
for almost every $t \in [0,T]$,
where $\mathcal{Z}$ is defined as in the proof of Proposition \ref{convergence}.
Since $\A_\eps (x,y)$, $(\xi,\xi_\Gamma)$ and
$\B \cdot \W$ belong to $L^1(0,T; \mathcal{Z}^*)$, we deduce that
$(x,y) \in C([0,T]; \mathcal{Z}^*)$ and then \eqref{first limit equation} holds for every $t \in [0,T]$.
Moreover, from a result due to Strauss \cite[Thm.~2.1]{strauss}, it follows that the solution is also weakly continuous in $\H$,
i.e.~$(x,y) \in C_w([0,T];\H)$.

\noindent It remains to show that $(\xi,\xi_\Gamma)$ belongs to $(\beta(x),\beta_\Gamma(y))$ a.e.~in $(0,T) \times D$. 
Let use the strong convergence of $(x_{\lambda'},y_{\lambda'}) \to (x,y)$ in $L^2(0,T;\H)$ 
to extract a subsequence (still denoted with $\lambda'$) so that both $(I + \lambda'\beta)^{-1}x_{\lambda'}$ and 
$(I + \lambda'\beta_\Gamma)^{-1}y_{\lambda'}$ would converge a.e. in $(0,T) \times D$ to $x$ and $y$, respectively. 
At this point remember that  $\beta_{\lambda'}(x_{\lambda'}) \in \beta\left( (I+\lambda\beta)^{-1}x_{\lambda'}\right)$ almost everywhere in $(0,T) \times D$ and that 
\begin{equation*}
\int_0^T\int_D \beta_{\lambda'}(x_{\lambda'})(I+\lambda'\beta)^{-1}x_{\lambda'} \leq 
\int_0^T\int_D \beta_{\lambda'}\left(x_{\lambda'} \right)x_{\lambda'} < K_\omega
\end{equation*}
because of Lemma \ref{ito} and the fact that $\omega\in\Omega'$. Then we are in position to apply Brezis'~lemma
\cite[Thm.~18, p.~126]{brezis3} and we get the identification $\xi \in \beta(x)$. 
The same reasoning holds for the monotone operator $\beta_\Gamma$ on the boundary and the claim is proved.

\noindent Finally, using the lower semicontinuity of the convex integrands for the weak convergence 
and Lemma~\ref{pathwise},
we have that the limit solution satisfies
\begin{equation*}
\begin{split}
\int_0^T \int_D \left( j(x) + j^*(\xi) \right) &\leq 
\liminf_{\lambda' \to 0} \int_0^T \int_D \left( j((I + \lambda'\beta)^{-1}x_{\lambda'}) + j^*(\beta_{\lambda'}(x_{\lambda'})) \right)\\
&=  \liminf_{\lambda' \to 0} \int_0^T \int_D \beta_{\lambda'}(x_{\lambda'}) (I + \lambda'\beta)^{-1}x_{\lambda'} \leq C_1,
\end{split}
\end{equation*}
\begin{equation*}
\begin{split}
\int_0^T \int_\Gamma \left( j_\Gamma(y) + j_\Gamma^*(\xi_\Gamma) \right)&\leq 
\liminf_{\lambda' \to 0} \int_0^T \int_\Gamma \left( j((I + \lambda'\beta_\Gamma)^{-1}y_{\lambda'}) 
+ j^*(\beta_{\Gamma,\lambda'}(y_{\lambda'})) \right)\\
&=  \liminf_{\lambda' \to 0} \int_0^T \int_\Gamma \beta_{\Gamma,\lambda'}(y_{\lambda'}) 
(I + \lambda'\beta_\Gamma)^{-1}y_{\lambda'} \leq C_2\,,
\end{split}
\end{equation*}
where $C_1,C_2$ are constants which depend only on $\omega$.

\textbf{Uniqueness.}
Here we show a conditional uniqueness result for equation \eqref{first limit equation} with $\omega$ fixed.
By contradiction, suppose that $(x_i,y_i,\xi_i,\xi_{\Gamma,i})$ satisfy equation \eqref{first limit equation} and
are such that $(x_i,y_i)\in L^\infty(0,T;\H)\cap L^2(0,T;\V_\eps)$, 
$(\xi_i, \xi_{\Gamma,i})\in L^1(0,T; L^1(D)\times L^1(\Gamma))$,
$j(x_i)+j^*(\xi_i)\in L^1((0,T)\times D)$, $j_\Gamma(y_i)+j_\Gamma^*(\xi_{\Gamma,i})\in L^1((0,T)\times \Gamma)$.
Setting $(\bar x, \bar y) = (x_1,y_1) - (x_2,y_2)$, $(\bar\xi,\bar\xi_\Gamma) = (\xi_1,\xi_{\Gamma,1}) -(\xi_2,\xi_{\Gamma,2})$
and $\bar\P:=\P(x_1,y_1)-\P(x_2,y_2)$ we have that
\begin{equation}\label{eq:uniq}
(\bar x(t),\bar y(t)) + \int_0^t \C_\eps(\bar x(s),\bar y(s))\,ds + \int_0^t\bar\P(s)\,ds+
 \int_0^t (\bar \xi(s),\bar \xi_\Gamma(s))\,ds = 0 \quad \forall\,t \in [0,T]
\end{equation}
and it is enough to show that $(\bar x, \bar y) = 0$ and $(\bar\xi,\bar\xi_\Gamma) = 0$ for every $t \in [0,T]$.
The idea is to smooth out the equation, recover some structure condition and then use strong convergence to pass to the 
limit in the approximation. To this end, we multiply the above equation by a power (high enough) of the resolvent of $\C_\eps$. 
Remember that the strong formulation of $\C_\eps$ is 
\[ \C_\eps(u,v):= (-\Delta u,  \partial_{\bf n} u - \eps \Delta_\Gamma  v) \]
and that we are able to show regularizing properties of the resolvent operator 
associated to $\C_\eps$, which are presented in detail in the Appendix.
In particular, we can prove (see Corollary \ref{cor_Ceps}) that there is $m \in \mathbb{N}$ so that 
$(I+ \delta \C_\eps)^{-m}$ is ultracontractive, i.e.~that it is well-defined, linear and continuous 
from $L^1(D)\times L^1(\Gamma)$ to 
$L^\infty(D)\times L^\infty(\Gamma)$. If we use the notation $(u,v)^\delta=(u^\delta,v^\delta):=(I+\delta\C_\eps)^{-m}(u,v)$
for any $(u,v)$ for which it makes sense and for any $\delta>0$,
then we have
\begin{equation*}
(\bar x,\bar y)^\delta(t) + \int_0^t \C_\eps(\bar x,\bar y)^\delta(s)\,ds + \int_0^t \bar\P^\delta(s) ds + 
\int_0^t (\bar \xi,\bar \xi_\Gamma)^\delta(s)\,ds = 0
\end{equation*}
for every $t\in[0,T]$.
Applying energy estimates (with $\omega$ still fixed) of the form $e^{-rt}\l| \cdot \r|_\H^2$, with $r>0$, we get
\begin{equation*}
\begin{split}
e^{-rt} \l|(\bar x,\bar y)^\delta(t)\r|^2_\H &+
2\int_0^te^{-rs}\l|\nabla \bar x^\delta(s)\r|_H^2\,ds+2\eps\int_0^te^{-rs}\l|\nabla_\Gamma \bar y^\delta(s)\r|^2_{H_\Gamma}\,ds\\
&+2r\int_0^te^{-rs}\l|(\bar x,\bar y)^\delta(s)\r|^2_\H\,ds+
2\int_0^te^{-rs}\left(\bar\P^\delta(s), (\bar x,\bar y)^\delta(s)\right)_\H\,ds\\
&+2 \int_0^t e^{-rs}\left((\bar\xi, \bar \xi_\Gamma)^\delta (s),(\bar x,\bar y)^\delta (s) \right)_\H\,ds = 0\,.
\end{split}
\end{equation*}
By the monotonicity of $\C_\eps$, the Lipschitz continuity of $\P$ and the Gronwall lemma we have
\beq\label{eq:approx_uniq}
e^{-rt} \l|(\bar x,\bar y)^\delta(t)\r|^2_\H 
+2 \int_0^t e^{-rs}\left((\bar\xi, \bar \xi_\Gamma)^\delta (s),(\bar x,\bar y)^\delta (s) \right)_\H\,ds\leq0\,.
\eeq
Thanks to Lemma \ref{lm_Ceps5} we know that $(\bar x,\bar y)^\delta(t) \to (\bar x,\bar y)(t)$ in $\H$
for all $t \in [0,T]$. This yields the convergence of $\l| (\bar x,\bar y)^\delta(t) \r|_\H^2$ towards $ \l| (\bar x,\bar y)(t) \r|_\H^2$. 
Concerning the second term, let us carefully rewrite it using the projection of the resolvent map defined in the Appendix. 
To shorten the notation here we denote by $(\bar x,\bar y)^\delta_i$ the projection on the $i$-th coordinate: 
$p_i \circ (I + \delta \C_\eps)^{-m}(\bar x,\bar y)$, $i = 1,2$ (the same for $(\bar \xi, \bar \xi_\Gamma)^\delta_i$). 
Hence we have
\begin{equation*}
\begin{split}
\left( (\bar x,\bar y)^\delta , (\bar \xi, \bar \xi_\Gamma)^\delta \right)_\H &= 
\left( (I + \delta \C_\eps)^{-m}(\bar x,\bar y) , (I + \delta \C_\eps)^{-m}(\bar \xi,\bar \xi_\Gamma) \right)_\H \\
&= \left( \left( (\bar x,\bar y)^\delta_1, (\bar x,\bar y)^\delta_2 \right) , 
\left( (\bar \xi, \bar \xi_\Gamma)^\delta_1, (\bar \xi, \bar \xi_\Gamma)^\delta_2 \right) \right)_\H \\
&= \big( (\bar x,\bar y)^\delta_1,(\bar \xi, \bar \xi_\Gamma)^\delta_1 \big)_H +
 \big( (\bar x,\bar y)^\delta_2, (\bar \xi, \bar \xi_\Gamma)^\delta_2 \big)_{H_\Gamma} \\
&= \int_D (\bar x,\bar y)^\delta_1 (\bar \xi, \bar \xi_\Gamma)^\delta_1 
+ \int_{\Gamma} (\bar x,\bar y)^\delta_2 (\bar \xi, \bar \xi_\Gamma)^\delta_2\,.
\end{split}
\end{equation*}  
Since $(\bar x,\bar y)^\delta \to (\bar x,\bar y)$ in $L^2(0,T;\H)$
and $(\bar\xi, \bar \xi_\Gamma)^\delta \to (\bar\xi, \bar \xi_\Gamma)$ in $L^1(0,T; L^1(D) \times L^1(\Gamma))$,
we can extract a subsequence of $\delta$, still denoted by the same symbol, 
such that $(\bar x,\bar y)^\delta \to (\bar x,\bar y)$ and $(\bar\xi, \bar \xi_\Gamma)^\delta \to (\bar\xi, \bar \xi_\Gamma)$
a.e.~in $(0,T) \times D\times \Gamma$.
The continuity of the projection maps $p_1, p_2$ yield the following convergence result:
\begin{equation*}
\begin{split}
(\bar x,\bar y)^\delta_1 (\bar\xi, \bar \xi_\Gamma)^\delta_1 &\longrightarrow \bar x\bar\xi
\quad \text{ a.e.~in }  (0,t) \times D\,, \\
(\bar x,\bar y)^\delta_2 (\bar\xi, \bar \xi_\Gamma)^\delta_2 &\longrightarrow \bar y \bar\xi_\Gamma
\quad \text{ a.e.~in }  (0,t) \times \Gamma\,.
\end{split}
\end{equation*}
Moreover, thanks to the assumptions (H1), (H2) or (H$3_{\eps>0}$)-(H$3_{\eps=0}$) on the compatibility of 
the potentials $j, j_\Gamma$ and Corollary \ref{cor_Ceps2} in the Appendix,
the families $(\bar x,\bar y)^\delta_1(\bar \xi, \bar \xi_\Gamma)^\delta_1$ and 
$(\bar x,\bar y)^\delta_2(\bar \xi, \bar \xi_\Gamma)^\delta_2$, 
with $i = 1,2$, are uniformly integrable: hence, using Vitali's theorem we infer that
$(\bar x,\bar y)^\delta_1 (\bar\xi, \bar \xi_\Gamma)^\delta_1 \rarr \bar x\bar\xi$ and
$(\bar x,\bar y)^\delta_2 (\bar\xi, \bar \xi_\Gamma)^\delta_2 \rarr \bar y \bar\xi_\Gamma$ in
$L^1((0,t)\times D)$ and $L^1((0,t)\times \Gamma)$, respectively.
Letting then $\delta\searrow0$ in \eqref{eq:approx_uniq}, we get that
\begin{equation*}
e^{-rt} \l|(\bar x,\bar y)(t)\r|^2_\H + 2 \int_0^t \int_D e^{-rs} \bar x\bar \xi
+ 2 \int_0^t \int_{\Gamma} e^{-rs}\bar y\bar \xi_\Gamma \leq 0\,.
\end{equation*}
From the monotonicity of $(\beta,\beta_\Gamma)$
 we deduce that $(\bar x, \bar y)(t) = 0$ for all $t \in [0,T]$. Then also $\int_0^t(\xi,\xi_\Gamma)(s)ds = 0$
for all $t\in[0,T]$ from equation \eqref{eq:uniq} and we get the uniqueness due to the arbitrariness of $t \in [0,T]$.

\textbf{Regularity in $\omega$.}
 So far we have worked with $\omega \in \Omega'$ fixed. Now we are interested in showing the regularity
 of the solution with respect to $\omega$, starting from the issue of measurability. 
We follow the same ideas developed in \cite{mar-scar}, which we briefly resume here for the convenience of the reader, 
with obvious adaptations to our setting. The key point is the uniqueness result obtained for the solution, which guarantees 
that the subsequence $\lambda'$ actually does not depend on $\omega \in \Omega'$. 
The proof relies on a standard argument: from any subsequence of $\lambda$ we can extract a further subsequence $\lambda'$,
 depending on $\omega$, such that all the convergences obtained in Proposition \ref{convergence} take place. 
 Since the limit is unique, then the same results holds for the entire sequence $\lambda$ which does not depend on $\omega$ 
 anymore.  

\noindent From Proposition \ref{convergence} we know that $(x_\lambda,y_\lambda)(\omega, \cdot) 
\rarr(x,y)(\omega,\cdot)$ strongly in $L^2(0,T;\H)$, which implies the convergence of $(x,y)(\omega,t)$ in  
$\mathbb{P}\otimes dt$-measure to the same limit. Then, along a subsequence we have convergence 
$\mathbb{P}\otimes dt$-a.e.~and this is enough to ensure the predictability of the limit $(x,y)$ 
(since $(x_\lambda,y_\lambda)$ is adapted with continuous trajectories).
Concerning the singular part $(\xi,\xi_\Gamma)$ we have to be more careful. First, we set
$\xi_\lambda:=\beta_\lambda(x_\lambda)$, $\xi_{\Gamma,\lambda}:=\beta_{\Gamma,\lambda(y_\lambda)}$ and
define for any $g \in L^\infty((0,T) \times D)$, $h \in L^\infty((0,T) \times \Gamma)$
\begin{align*}
F_\lambda(\omega):= \int_0^T \int_D \xi_\lambda(\omega,s,z)g(s,z)\,dzds, \quad 
F(\omega):= \int_0^T \int_D \xi(\omega,s,z)g(s,z)\,dzds\,, \\
F_{\Gamma,\lambda}(\omega):= \int_0^T \int_\Gamma \xi_{\Gamma,\lambda}(\omega,s,\zeta)h(s,\zeta)\,d\zeta ds\,, \quad
 F_{\Gamma}(\omega):= \int_0^T \int_\Gamma \xi_{\Gamma}(\omega,s,\zeta)h(s,\zeta)\,d\zeta ds\,.
\end{align*}
Then, since $\xi_\lambda\rarrw\xi$ and $\xi_{\Gamma,\lambda}\rarrw\xi_\Gamma$
in $L^1((0,T)\times D)$ and $L^1((0,T)\times \Gamma)$, respectively, we have that
$F_\lambda\rarr F$ and $F_{\Gamma,\lambda}\rarr F_\Gamma$ $\Pi$-almost surely in $\Omega$.
What we are going to show now is the following. We first prove that
$F_\lambda$ and $F_{\Gamma,\lambda}$ converge weakly in $L^1(\Omega)$ to $F$ and $F_\Gamma$, respectively,
so that 
$(\xi_\lambda,\xi_{\Gamma,\lambda})\rarrw(\xi,\xi_\Gamma)$ in $L^1(\Omega; L^1(0,T; L^1(D)\times L^1(\Gamma)))$.
Then, using the
Mazur lemma, we will infer measurability of the limit $(\xi, \xi_\Gamma)$.

\noindent Let us show some details for $F_\lambda$; the same technique can be adopted for $F_{\Gamma, \lambda}$.
Firstly, for any $l\in L^\infty(\Omega)$, setting $j_0(\cdot):= j^*(\cdot / M)$ and 
$M= 1/ \big[ (\l| g \r|_{L^\infty((0,T) \times D)} \vee 1)(\l| l \r|_{L^\infty( \Omega)} \vee 1) \big]$,
we use Jensen's inequality to get 
\begin{equation*}
\E j_0(F_\lambda l) = \E j_0 \left( \int_0^T \int_D \xi_\lambda(\omega,s,z)g(s,z)l(\omega)\, dz ds \right)
\lesssim \E \int_0^T \int_D j^*(\xi_\lambda(\omega,s,z))\, dz ds\,.
\end{equation*}
The last term is bounded due to Lemma~\ref{ito}: then by the de la Vall\'ee Poussin criterion 
the sequence $F_\lambda l$ is uniformly integrable. This yields the strong convergence
$F_\lambda l\rarr F l$ in $L^1(\Omega)$ thanks to Vitali's theorem. From the arbitrariness of the function $l$
we get weak convergence in $L^1(\Omega)$ of $(F_\lambda)_\lambda$, thus the convergence of $\xi_\lambda$ to $\xi$  
in $L^1(\Omega \times (0,T) \times D)$. Applying Mazur's lemma we extract a convex combination 
$(\tilde \xi_\lambda)_\lambda$ of $(\xi_\lambda)_\lambda$ 
which converge strongly to $\xi$ in $L^1(\Omega\times(0,T)\times D)$.
To conclude note that $\tilde \xi_\lambda$ are predictable 
(as finite convex combination of predictable processes), hence the limit $\xi$ is a predictable $L^1(D)$-valued process.   

\noindent Once we have the measurability of the limit solution $(x,y,\xi, \xi_\Gamma)$, 
we are interested in showing some estimates in expectation. To this end, we just use lower semicontinuity 
of the norms with respect to the weak convergence and Fatou's lemma:
\begin{align*}
\E \l|  (x,y) \r|^2_{L^2(0,T; \V_\eps)} \leq \E \liminf_{\lambda \to 0} \l|  (x_\lambda,y_\lambda) \r|^2_{L^2(0,T; \V_\eps)} \leq
 \liminf_{\lambda \to 0} \E \l| (x_\lambda,y_\lambda) \r|^2_{L^2(0,T; \V_\eps)}\,, \\
\E \l|  (x,y) \r|^2_{L^\infty(0,T; \H)} \leq \E \liminf_{\lambda \to 0} \l|  (x_\lambda,y_\lambda) \r|^2_{L^\infty(0,T; \H)} \leq
 \liminf_{\lambda \to 0} \E \l| (x_\lambda,y_\lambda) \r|^2_{L^\infty(0,T; \H)}\,, \\
\E \l|  \xi \r|^2_{L^1(0,T; L^1(D))} \leq \E \liminf_{\lambda \to 0} \l|  \xi_\lambda \r|^2_{L^1(0,T; L^1(D))} \leq
 \liminf_{\lambda \to 0} \E \l| \xi_\lambda \r|^2_{L^1(0,T; L^1(D))}\,, \\
\E \l|  \xi_\Gamma \r|^2_{L^1(0,T; L^1(\Gamma))} \leq 
\E \liminf_{\lambda \to 0} \l|  \xi_{\Gamma,\lambda} \r|^2_{L^1(0,T; L^1(\Gamma))} 
\leq \liminf_{\lambda \to 0} \E \l| \xi_{\Gamma,\lambda} \r|^2_{L^1(0,T; L^1(\Gamma))}\,.
\end{align*}
Thanks to Lemma \ref{ito} and Proposition \ref{convergence} the right hand side of all the above inequalities is bounded
uniformly in $\lambda$, hence 
\begin{align*}
(x,y) \in L^2\left(\Omega; L^\infty(0,T; \H)\right) \cap L^2\left(\Omega\times(0,T); \V_\eps\right),\\
  \xi \in L^1\left(\Omega\times(0,T)\times D\right), \qquad \xi_\Gamma \in L^1\left(\Omega\times(0,T)\times \Gamma\right),
\end{align*} 
and we have the required regularity.  
Finally, as we did at the beginning of Subsection \ref{limit_problem},
the weak lower semicontinuity of the convex integrands, Lemma \ref{ito} and the fact that the subsequence $\lambda$
does not depend on $\omega$ ensure also that
\[
  j(x)+j^*(\xi)\in L^1(\Omega\times(0,T)\times D)\,, \qquad 
  j(y)+j_\Gamma^*(\xi_\Gamma)\in L^1(\Omega\times(0,T)\times \Gamma)\,.
\]
This completes the proof of existence of solutions for the problem with additive and more regular noise.

\subsection{Continuous dependence on the data}
\label{cont_dep}

Here we prove a continuous dependence result for equation \eqref{additive} for every $\eps\geq0$.

We consider $(x_0^i,y_0^i)\in L^2(\Omega,\f_0,\Pi; \H)$ and $\B_i\in L^2(\Omega; L^2(0,T; \cL_2(\U,\H)))$ for $i=1,2$:
let $(x_i,y_i,\xi_i,\xi_{\Gamma,i})$ be any corresponding strong solutions to \eqref{additive},
$i=1,2$. 
Setting $(x,y):=(x_1,y_1)-(y_1,y_2)$, $(\xi,\xi_\Gamma):=(\xi_1,\xi_{\Gamma,1})-(\xi_2,\xi_{\Gamma,2})$, $\eta:=\P(x_1,y_1)-\P(x_2,y_2)$,
$x_0:=x_0^1-x_0^2$, $y_0:=y_0^1-y_0^2$ and $\B:=\B_1-\B_2$, we have
\[
  (x(t),y(t)) + \int_0^t\C_\eps(x(s),y(s))\,ds + \int_0^t\eta(s)\,ds
  +\int_0^t(\xi(s),\xi_\Gamma(s))\,ds=(x_0,y_0) + \int_0^t\B(s)\,d\W_s
\]
for every $t\in[0,T]$, $\Pi$-almost surely. Now, the idea is to proceed as in 
Subsection \ref{limit_problem} when we proved uniqueness for the limit problem:
from now on, we refer to the Appendix for any 
useful properties on the resolvent of $\C_\eps$.
If $m\in\En$ is given by Corollary \ref{cor_Ceps}, using again the notation
$(h_1,h_2)^\delta=(h_1^\delta,h_2^\delta):=(I+\delta\C_\eps)^{-m}h$ for any $h$ for which it makes sense, we have
\[
  (x,y)^\delta(t) + \int_0^t\C_\eps(x,y)^\delta(s)\,ds + \int_0^t\eta^\delta(s)\,ds
  +\int_0^t(\xi,\xi_\Gamma)^\delta(s)\,ds = (x_0,y_0)^\delta + \int_0^t\B^\delta(s)\,d\W_s\,.
\]
Thanks to the smoothing properties of $(I+\delta\C_\eps)^{-m}$, we can apply the classical It\^o's formula
for $e^{-rt}\l|\cdot\r|^2_\H$ to get
\[
  \begin{split}
  e^{-rt} &\l|(x,y)^\delta(t)\r|^2_\H +
  2\int_0^te^{-rs}\l|\nabla x^\delta(s)\r|_H^2\,ds+2\eps\int_0^te^{-rs}\l|\nabla_\Gamma y^\delta(s)\r|^2_{H_\Gamma}\,ds\\
  &+2r\int_0^te^{-rs}\l|(x,y)^\delta(s)\r|^2_\H\,ds+
  2\int_0^te^{-rs}\left(\eta^\delta(s),(x,y)^\delta(s)\right)_\H\,ds\\
  &+2 \int_0^t e^{-rs}\left((\xi, \xi_\Gamma)^\delta (s),(x,y)^\delta (s) \right)_\H\,ds \\
  &=\l|(x_0,y_0)^\delta\r|^2_\H
  +\int_0^te^{-rs}\l|\B^\delta(s)\r|^2_{\cL_2(\U,\H)}\,ds + 2\int_0^te^{-rs}(x,y)^\delta(s)\B^\delta(s)\,d\W_s\,.
  \end{split}
\]
We want to pass to the limit as $\delta\searrow0$ in the previous expression and
obtain in this way an It\^o's formula for the limit processes.
To this aim, the terms on the left hand side have already been handled in Subsection \ref{limit_problem}
when dealing with uniqueness of the limit problem.
Moreover, by virtue of Lemma \ref{lm_Ceps5}, the dominated convergence theorem
and the ideal property of $\cL_2(\U,\H)$ in $\cL(\U,\H)$ we have that 
$\l|(x_0,y_0)^\delta\r|^2_\H\rarr\l|(x_0,y_0)\r|^2_\H$
and $\int_0^te^{-rs}\l|\B^\delta(s)\r|^2_{\cL_2(\U,\H)}\,ds\rarr \int_0^te^{-rs}\l|\B(s)\r|^2_{\cL_2(\U,\H)}\,ds$
as $\delta\searrow0$. Hence, we only need to check the convergence of the stochastic integral:
note that, using the notation $f^r$ to indicate the process $f^r(t):=e^{-\frac{r}{2}t}f(t)$ for any process $f$,
thanks to the Burkholder-Davis-Gundy inequality we have
\[
  \begin{split}
  &\exval\sup_{t\in[0,T]}\left|\int_0^te^{-rs}\left[(x,y)^\delta(s)\B^\delta(s)-(x,y)(s)\B(s)\right]\,d\W_s\right|\\
  &\lesssim
  \exval\left[\left(\int_0^Te^{-2rs}\l|(x,y)^\delta(s)\r|^2_\H\l|\B^\delta(s)-\B(s)\r|^2_{\cL_2(\U,\H)}\,ds\right)^{1/2}\right]\\
  &\qquad\qquad+\exval\left[\left(\int_0^Te^{-2rs}\l|(x,y)^\delta(s)-(x,y)(s)\r|^2_\H
  \l|\B(s)\r|^2_{\cL_2(\U,\H)}\,ds\right)^{1/2}\right]\\
  &\leq\l|(x,y)^r\r|_{L^2(\Omega; L^\infty(0,T; \H))}\l|\B^{\delta,r}-\B^r\r|_{L^2(\Omega; L^2(0,T; \cL_2(\U,\H)))}\\
  &\qquad\qquad+\l|(x,y)^{\delta,r}-(x,y)^r\r|_{L^2(\Omega; L^\infty(0,T; \H))}\l|\B^r\r|_{L^2(\Omega; L^2(0,T; \cL_2(\U,\H)))}
  \end{split}
\]
and the last term converges to $0$ as $\delta\searrow0$ by Lemma \ref{lm_Ceps5}, the dominated convergence
theorem and the fact that $\B^{\delta,r}\rarr\B^r$ in $L^2(\Omega; L^2(0,T; \cL_2(\U,\H)))$. Consequently,
taking everything into account, we infer that the following It\^o's formula holds:
\[
  \begin{split}
  &\l|(x,y)^r(t)\r|^2_\H +
  2\int_0^t\l|\nabla x^r(s)\r|_H^2\,ds+2\eps\int_0^t\l|\nabla_\Gamma y^r(s)\r|^2_{H_\Gamma}\,ds
  +2r\int_0^t\l|(x,y)^r(s)\r|^2_\H\,ds\\
  &+2\int_0^t\left(\eta^r(s),(x,y)^r(s)\right)_\H\,ds+
  2 \int_0^t\int_D\xi^r(s)x^r(s)\,ds + 2\int_0^t\int_\Gamma\xi_\Gamma^r(s)y^r(s)\,ds \\
  &=\l|(x_0,y_0)\r|^2_\H
  +\int_0^t\l|\B^r(s)\r|^2_{\cL_2(\U,\H)}\,ds + 2\int_0^t(x,y)^r(s)\B^r(s)\,d\W_s\,.
  \end{split}
\]
Using the Lipschitz continuity of $\P$, the Gronwall lemma,
the monotonicity of $(\beta,\beta_\Gamma)$ and the weak coercivity of $\C_\eps$,
taking supremum in time and expectations we get
\[
  \begin{split}
  &\l|(x,y)^r(t)\r|^2_{L^2(\Omega; L^\infty(0,T; \H))} + r \l|(x,y)^r\r|_{L^2(\Omega; L^2(0,T; \H))}^2+
  \l|(x,y)^r\r|_{L^2(\Omega; L^2(0,T; \V_\eps))}^2 \\
  &\lesssim\l|(x_0,y_0)\r|^2_{L^2(\Omega;\H)} + \l|\B^r\r|^2_{L^2(\Omega; L^2(0,T;\cL_2(\U,\H)))}+
  \exval\sup_{t\in[0,T]}\left|\int_0^t(x,y)^r(s)\B^r(s)\,d\W_s\right|\,.
  \end{split}
\]
Estimating the last term as in the proof of Lemma \ref{ito}, we deduce that 
\beq
  \label{continuous}
  \begin{split}
  \l|(x,y)^r(t)\r|_{L^2(\Omega; L^\infty(0,T; \H))} &+ \sqrt{r} \l|(x,y)^r\r|_{L^2(\Omega; L^2(0,T; \H))}+
  \l|(x,y)^r\r|_{L^2(\Omega; L^2(0,T; \V_\eps))} \\
  &\lesssim\l|(x_0,y_0)\r|_{L^2(\Omega;\H)} + \l|\B^r\r|_{L^2(\Omega; L^2(0,T;\cL_2(\U,\H)))}\,,
  \end{split}
\eeq
where the desired continuous dependence relation for the problem with additive noise
follows choosing $r=0$.
This implies that the solution $(x,y,\xi,\xi_\Gamma)$ built as limit of Yosida approximations
of the problem is indeed the unique strong solution to the problem.

\subsection{Extension to general additive noise}
\carlo{Here we conclude the proof of existence of a solution to \eqref{additive}, with general additive noise: namely,
we remove the hypothesis \eqref{B_strong} and we assume just that
\[
  \B \in L^2\left(\Omega; L^2(0,T; \H)\right)\,.
\]
To do it, we follow the same ideas as in \cite[Prop.~5.1]{mar-scar}.}
Using the notation of the Appendix, for any $\delta \in (0,1)$ and $m$ given 
by Corollary~\ref{cor_Ceps} we have
\[
  \B^\delta :=(I+\delta\C_\eps)^{-m}\B \in L^2\left(\Omega; L^2(0,T; \Z)\right)
\]
for a certain Hilbert space $\Z \embed \V_\eps \cap L^\infty(D)\times L^\infty(\Gamma)$.
Hence, the problem with additive noise $\B^\delta$ admits a unique solution $(x^\delta,y^\delta, \xi^\delta, \xi_\Gamma^\delta)$
for what we have already proved.

Since by contraction of the resolvent
$\l|\B^\delta\r|_{L^2\left(\Omega; L^2(0,T; \H)\right)}\leq \l|\B\r|_{L^2\left(\Omega; L^2(0,T; \H)\right)}$
for every $\delta\in(0,1)$, by Lemma~\ref{ito} and the weak lower semicontinuity of the norms we have that
\begin{gather*}
\l|(x^\delta, y^\delta)\r|^2_{L^2(\Omega; L^\infty(0,T; \H))\cap L^2(\Omega; L^2(0,T; \V_\eps))}
+\l|\xi^\delta x^\delta\r|_{L^1(\Omega\times(0,T)\times D)} + 
\l|\xi^\delta_\Gamma y^\delta\r|_{L^1(\Omega\times(0,T)\times D)}\leq K
\end{gather*}
for a positive constant $K$, for every $\delta \in (0,1)$. Moreover,
by the continuous dependence result proved in the previous section, 
for every $\delta, \eta\in(0,1)$ we have
\[
\l|(x^\delta,y^\delta)-(x^\eta, y^\eta)\r|_{L^2(\Omega; L^\infty(0,T; \H))\cap L^2(\Omega; L^2(0,T; \V_\eps))}
\lesssim \l|\B^\delta-\B^\eta\r|_{L^2(\Omega; L^2(0,T;\H))}\to 0 
\]
as $\delta,\eta\searrow0$, so that $\{(x^\delta,y^\delta)\}_\delta$ is Cauchy in 
$L^2(\Omega; L^\infty(0,T; \H))\cap L^2(\Omega; L^2(0,T; \V_\eps))$. Now, recalling that 
\[
  \xi^\delta x^\delta = j(x^\delta) + j^*(\xi^\delta)\,, \quad
  \xi_\Gamma^\delta y^\delta = j_\Gamma(y^\delta) + j_\Gamma^*(\xi_\Gamma^\delta)\,,
\]
since $j^*$ and $j_\Gamma^*$ are superlinear, the de la Vall\'ee-Poussin and Dunford-Pettis theorems
ensure that $\{\xi^\delta\}_\delta$ and $\{\xi_\Gamma^\delta\}_\delta$ are relatively weakly compact in
$L^1(\Omega\times(0,T)\times D)$ and $L^1(\Omega\times(0,T)\times\Gamma)$, respectively.
Taking this information into account, we deduce that (along a subsequence)
\begin{align*}
  (x^\delta,y^\delta)\to (x,y) \qquad&\text{in } L^2(\Omega; L^\infty(0,T; \H))\cap L^2(\Omega; L^2(0,T; \V_\eps))\,,\\
  \xi^\delta \rarrw \xi\qquad&\text{in } L^1(\Omega\times(0,T)\times D)\,,\\
  \xi_\Gamma^\delta \rarrw\xi_\Gamma\qquad&\text{in } L^1(\Omega\times(0,T)\times \Gamma)\,.
\end{align*}
Using these convergences, the weak lower semicontinuity of the convex integrands 
and the strong-weak closure of the maximal monotone graphs as in the passage to the limit
in $\lambda$, it is not difficult to check that $(x,y,\xi, \xi_\Gamma)$ is a strong solution
for the problem with additive noise.

\subsection{Well-posedness with multiplicative noise}

We collect here the conclusion of the proof of Theorem \ref{teo1}, showing that 
the original problem \eqref{prob} is well-posed also with multiplicative noise.

Given a progressively measurable process $(z,w)\in L^2(\Omega; L^2(0,T; \H))$, thanks to the hypotheses
on $\B$, we have that $\B(\cdot,\cdot, z,w)\in L^2(\Omega; L^2(0,T; \cL_2(\U,\H)))$ and is progressively 
measurable. By the well-posedness result with additive noise proved in the previous sections, the problem
\[
  d(x_t,y_t) + \A_\eps(x_t,y_t)\,dt + \gamma(x_t,y_t)\,dt \ni \B(t,z_t,w_t)\,d\W_t\,, \quad (x(0),y(0))=(x_0,y_0)\,,
\]
is well-posed. Hence, it is well-defined the map
\begin{align*}
  \mathcal{S}_\eps: L^2\left(\Omega; L^2(0,T; \H)\right)&\rarr L^2\left(\Omega; L^\infty(0,T; \H)\right)\cap 
  L^2\left(\Omega; L^2(0,T;\V_\eps)\right)\,,\\
  \mathcal{S}_\eps:(z,w)&\mapsto(x,y)\,.
\end{align*}
It is clear that any fixed point $(x,y)$ for $\mathcal{S}_\eps$ (together with its corresponding $(\xi,\xi_\Gamma)$)
is a strong solution to \eqref{prob}. 
If we introduce for any $p\in[1,+\infty]$ the norms on $L^p(0,T)$
given by $\l|f\r|_{L^p_r(0,T)}:=\l|f^r\r|_{L^p(0,T)}$, for $f\in L^p(0,T)$,
using \eqref{continuous} and the Lipschitz continuity of $\B$, it is immediate to check that
\[
  \begin{split}
  \l|\mathcal{S}_\eps(z_1,w_1)-\mathcal{S}_\eps(z_2,w_2)\r|_{L^2(\Omega; L^2_r(0,T;\H))}&\lesssim
  \frac{1}{\sqrt{r}}\l|\B(z_1,w_1)-\B(z_2,w_2)\r|_{L^2(\Omega; L^2_r(0,T; \cL_2(\U,\H)))}\\
  &\leq\frac{1}{\sqrt{r}}\l|(z_1,w_1)-(z_2,w_2)\r|_{L^2(\Omega; L^2_r(0,T; \H))}\,,
  \end{split}
\]
where the implicit constant is independent of $r$. Consequently, there exists $r$ large enough
such that $\mathcal{S}_\eps$ is a strict contraction on $L^2(\Omega; L^2_r(0,T; \H))$. By Banach 
fixed point theorem, there exists a unique fixed point $(x,y)$ for $\mathcal{S}_\eps$:
this implies that \eqref{prob} has a unique strong
solution.

Finally, the last thing we have to prove is that the solution map
\begin{align*}
  \Lambda_\eps:L^2\left(\Omega,\f_0,\Pi; \H\right)&\rarr L^2\left(\Omega; L^\infty(0,T; \H)\right)\cap 
  L^2\left(\Omega; L^2(0,T;\V_\eps)\right)\,,\\
  \Lambda_\eps:(x_0,y_0)&\mapsto(x,y)\,,
\end{align*}
where $(x,y)$ is the unique solution to \eqref{prob}, is Lipschitz continuous. To this aim,
given $(x_0^i,y_0^i)\in L^2(\Omega,\f_0,\Pi; \H)$ and setting $(x_i,y_i):=\Lambda_\eps(x_0^i,y_0^i)$
for $i=1,2$,
rewriting \eqref{continuous} with the choice $\B_i:=\B(x_i,y_i)$, $i=1,2$, and
using the Lipschitz continuity of $\B$, we have 
\[
  \begin{split}
  &\l|(x_1,y_1)-(x_2,y_2)\r|_{L^2(\Omega; L^\infty_r(0,T;\H))\cap L^2(\Omega; L^2_r(0,T; \V_\eps))}
  +\sqrt{r}\l|(x_1,y_1)-(x_2,y_2)\r|_{L^2(\Omega; L^2_r(0,T; \H))}\\
  &\qquad\lesssim\l|(x_0^1,y_0^1)-(x_0^2,y_0^2)\r|_{L^2(\Omega; \H)}+
  \l|\B(x_1,y_1)-\B(x_2,y_2)\r|_{L^2(\Omega; L^2_r(0,T; \cL_2(\U,\H)))}\\
  &\qquad\lesssim\l|(x_0^1,y_0^1)-(x_0^2,y_0^2)\r|_{L^2(\Omega; \H)}+
  \l|(x_1,y_1)-(x_2,y_2)\r|_{L^2(\Omega; L^2_r(0,T; \H))}
  \end{split}
\]
where again the implicit constant does not depend on $r$. Consequently, the assertion 
follows choosing $r$ large enough and from the fact that $\l|\cdot\r|_{L^p_r(0,T)}$ is equivalent 
to the usual norm of $L^p(0,T)$ for every $r\geq0$.


\section{The asymptotic behaviour as $\eps\searrow0$}
\setcounter{equation}{0}
\label{asympt}

This last section is devoted to the proof of the asymptotic result as $\eps\searrow0$ contained in Theorem \ref{teo2}.
Throughout the section, $(x_\eps,y_\eps, \xi_\eps, \xi_{\Gamma,\eps})$ and $(x,y,\xi,\xi_\Gamma)$ are the unique solutions
to \eqref{prob} with additive noise $\B$ in the cases $\eps>0$ and $\eps=0$, respectively.
Moreover, for any $\delta\in(0,1)$, $(x_\eps^\delta,y^\delta_\eps, \xi^\delta_\eps, \xi^\delta_{\Gamma,\eps})$
and $(x^\delta,y^\delta,\xi^\delta,\xi^\delta_\Gamma)$ are the unique solutions
to \eqref{prob} with additive noise $(I+\delta\C_\eps)^{-m}\B$ in the cases $\eps>0$ and $\eps=0$, respectively,
where $m$ is given by Corollary~\ref{cor_Ceps}.
From the previous section, we know that, as $\delta\searrow0$,
\begin{align*}
  (x^\delta_\eps,y^\delta_\eps)\to (x_\eps,y_\eps)\,, \quad 
  (x^\delta,y^\delta) \to (x,y) \qquad&\text{in } 
  L^2(\Omega; L^\infty(0,T; \H))\cap L^2(\Omega; L^2(0,T; \V_\eps))\,,\\
  \xi^\delta_\eps \rarrw \xi_\eps\,,
  \quad \xi^\delta\rarrw\xi\qquad&\text{in } L^1(\Omega\times(0,T)\times D)\,,\\
  \xi_{\Gamma,\eps}^\delta \rarrw\xi_{\Gamma,\eps}
  \quad\xi_{\Gamma}^\delta\rarrw\xi_\Gamma\qquad&\text{in } L^1(\Omega\times(0,T)\times \Gamma)\,.
\end{align*}
Consequently, it is sufficient to prove Theorem~\ref{teo2} for $\B^\delta$, with $\delta\in(0,1)$ fixed.
For this reason, it is not restrictive to assume \eqref{B_strong} here.

By definition of strong solution, we know that 
\beq\label{sol_eps}
\begin{split}
  (x_\eps(t), y_\eps(t)) &+ \int_0^t\C_\eps(x_\eps(s), y_\eps(s))\,ds 
  +\int_0^t \P(x_\eps(s),y_\eps(s))\,ds + \int_0^t(\xi_\eps(s), \xi_{\Gamma,\eps}(s))\,ds\\
  &=(x_0, y_0) + \int_0^t\B(s)\,d\W_s \quad\forall\,t\in[0,T]\,, \quad\Pi\text{-a.s.}\,.
\end{split}
\eeq
Proceeding as in Section \ref{cont_dep} with the choice $r=0$, the following It\^o's formula holds
for every $t\in[0,T]$, $\Pi$-almost surely:
\[
\begin{split}
  &\l|(x_\eps,y_\eps)(t)\r|^2_\H +
  2\int_0^t\l|\nabla x_\eps(s)\r|_H^2\,ds+2\eps\int_0^t\l|\nabla_\Gamma y_\eps(s)\r|^2_{H_\Gamma}\,ds\\
  &+2\int_0^t\left(\P(x_\eps(s),y_\eps(s)),(x_\eps,y_\eps)(s)\right)_\H\,ds+
  2 \int_0^t\int_D\xi_\eps(s) x_\eps(s)\,ds + \int_0^t\int_\Gamma \xi_{\Gamma,\eps}(s)y_\eps(s)\,ds \\
  &=\l|(x_0,y_0)\r|^2_\H
  +\int_0^t\l|\B(s)\r|^2_{\cL_2(\U,\H)}\,ds + 2\int_0^t(x_\eps,y_\eps)(s)\B(s)\,d\W_s\,.
\end{split}
\]
Taking supremum in time and expectations, estimating the stochastic integral as in the proof of 
Lemma \ref{ito} and using the Gronwall lemma together with the 
lipschitzianity of $\P$, we easily deduce that there exists $C>0$, independent of $\eps$, such that
\begin{gather*}
  \l|(x_\eps,y_\eps)\r|_{L^2(\Omega; L^\infty(0,T; \H))}\leq C\,,\\
  \l|\nabla x_\eps\r|_{L^2(\Omega; L^2(0,T; H))}\leq C\,, \qquad
  \eps^{1/2}\l|\nabla_\Gamma y_\eps\r|_{L^2(\Omega; L^2(0,T; H_\Gamma))}\leq C\,,\\
  \l|\xi_\eps x_\eps\r|_{L^1(\Omega\times(0,T)\times D)}\leq C\,, \qquad
  \l|\xi_{\Gamma,\eps}y_\eps\r|_{L^1(\Omega\times(0,T)\times \Gamma)}\leq C\,.
\end{gather*}

Similarly, proceeding as in Lemma~\ref{pathwise} and owing to \eqref{B_strong}, we infer that there
exists $\Omega'\in\f$ with $\Pi(\Omega')=1$ such that, for every $\omega\in\Omega'$, there is $C_\omega>0$
(independent of $\eps\in(0,1)$) such that
\begin{gather*}
  \l|(x_\eps,y_\eps)(\omega)\r|_{L^\infty(0,T; \H)}\leq C_\omega\,,\\
  \l|\nabla x_\eps(\omega)\r|_{L^2(0,T; H)}\leq C_\omega\,, \qquad
  \eps^{1/2}\l|\nabla_\Gamma y_\eps(\omega)\r|_{L^2(0,T; H_\Gamma)}\leq C_\omega\,,\\
  \l|\xi_\eps(\omega) x_\eps(\omega)\r|_{L^1((0,T)\times D)}\leq C_\omega\,, \qquad
  \l|\xi_{\Gamma,\eps}(\omega)y_\eps(\omega)\r|_{L^1((0,T)\times \Gamma)}\leq C_\omega\,.
\end{gather*}
Since $(\xi_\eps,\xi_{\Gamma,\eps})\in\gamma(x_\eps,y_\eps)$ almost everywhere,
we have
\[
j^*(\xi_\eps)\leq j(x_\eps)+j^*(\xi_\eps)=\xi_\eps x_\eps\,, \qquad
j_\Gamma^*(\xi_{\Gamma,\eps})\leq j(y_\eps)+j_\Gamma^*(\xi_{\Gamma,\eps})=\xi_{\Gamma,\eps} y_\eps\,,
\]
so that the families $j^*(\xi_\eps)$ and $j^*_\Gamma(\xi_{\Gamma,\eps})$ are uniformly bounded
in $L^1((0,T)\times D)$ and $L^1((0,T)\times \Gamma)$, respectively.
Furthermore, since
\[
  \frac{d}{dt}\left((x_{\eps},y_{\eps})-\B\cdot\W\right) + 
  \A_{\eps}(x_{\eps},y_{\eps}) + 
  (\xi_{\eps},\xi_{\Gamma,{\eps}})  = 0\,,
\]
if $\mathcal{Z}$ is as in the proof of Proposition \ref{convergence},
by difference we also infer that 
\[
  \l|\frac{d}{dt}\left((x_{\eps},y_{\eps})(\omega)-\B\cdot\W(\omega)\right)\r|_{L^1(0,T; \mathcal{Z}^*)}\leq C_\omega\,.
\]
We deduce that for every $\omega\in\Omega'$ there is a subsequence $\eps'=\eps'(\omega)$ along which 
we have
\begin{align}
  \label{c1'}
  &(x_{\eps'},y_{\eps'})\rarr(\tilde x,\tilde y) &&\text{in } L^2(0,T; \H)\,,\\
  \label{c2'}
  &(x_{\eps'},y_{\eps'})\weakstar(\tilde x,\tilde y) &&\text{in } L^\infty(0,T; \H)\,,\\
  \label{c3'}
  &(x_{\eps'},y_{\eps'})\rarrw(\tilde x,\tilde y) &&\text{in } L^2(0,T; \V_0)\,,\\
  \label{c4'}
  &(\xi_{\eps'},\xi_{\Gamma,{\eps'}})\rarrw (\tilde\xi, \tilde\xi_\Gamma) &&\text{in }
  L^1(0,T; L^1(D)\times L^1(\Gamma))\,.
\end{align}
Moreover, it is also clear that
\beq
  \label{c5'}
  \eps\nabla_\Gamma y_{\eps}\rarr0 \qquad\text{in } L^2\left(\Omega; L^2(0,T; H_\Gamma)\right)\,.
\eeq

At this point, repeating exactly the same argument contained in Section \ref{limit_problem},
we infer that $\eps'$ is independent of $\omega$ and that the limit processes satisfy
\begin{gather*}
  (\tilde x,\tilde y) \in L^2\left(\Omega;L^\infty(0,T;\H)\right)\cap L^2\left(L^2(0,T;\V_0)\right)\,,\\
  (\tilde\xi,\tilde\xi_\Gamma) \in L^1\left(\Omega;L^1(0,T; L^1(D)\times L^1(\Gamma))\right)\,,\\
  j(\tilde x) + j^*(\tilde\xi) \in L^1\left(\Omega\times(0,T)\times D\right)\,, \qquad
  j_\Gamma(\tilde y) + j^*_\Gamma(\tilde\xi_\Gamma) \in L^1\left(\Omega\times(0,T)\times \Gamma\right)\,,\\
  \begin{split}
  (\tilde x(t), \tilde y(t)) &+ \int_0^t\C_0(\tilde x(s), \tilde y(s))\,ds 
  +\int_0^t \P(\tilde x(s), \tilde y(s))\,ds + \int_0^t(\tilde \xi(s), \tilde \xi_\Gamma(s))\,ds\\
  &=(x_0, y_0) + \int_0^t\B(s)\,d\W_s \quad\forall\,t\in[0,T]\,, \quad\Pi\text{-a.s.}\,.
\end{split}
\end{gather*}
Since the problem \eqref{additive} is well-posed for $\eps=0$ (in particular, it admits 
a unique solution), we deduce that 
$(\tilde x, \tilde y)=(x,y)$ and $(\tilde\xi,\tilde\xi_\Gamma)=(\xi,\xi_\Gamma)$.

The last thing that we have to show is that the convergences \eqref{c1}--\eqref{c5} hold.
To this aim, note that \eqref{c1} and \eqref{c5} coincide with \eqref{c1'} and \eqref{c5'}, respectively.
Moreover, \eqref{c2}--\eqref{c4} follow from 
\eqref{c2'}--\eqref{c4'} using the fact that 
$\eps'$ is independent of $\omega$ and Vitali's convergence theorem,
via a similar argument to the one performed 
in Section \ref{limit_problem} when dealing with the regularity in $\omega$.


\appendix
\section{An auxiliary result}
\label{A_app}

Let $\eps\geq0$ be fixed. We introduce the operator
\beq
  \label{def_Ceps}
  \mathcal{C}_\eps:\H\rarr\H\,, \quad \mathcal{C}_\eps(u,v):=(-\Delta u, \partial_{\bf n}u - \eps\Delta_\Gamma v)\,,
   \quad (u,v)\in D(\mathcal{C}_\eps)\,,
\eeq
where
\beq
  \label{dom_Ceps}
  D(\mathcal{C}_\eps):=
  \begin{cases}
  \left\{(u,v)\in H^2(D)\times H^{2}(\Gamma): v=\tau u\right\}\quad&\text{if } \eps>0\,,\\
  \left\{(u,v)\in H^{3/2}(D)\times H^{1}(\Gamma): \Delta u\in L^2(D),\; v=\tau u\right\} &\text{if } \eps=0\,.
  \end{cases}
\eeq
It is clear $\C_\eps$ is a well-defined linear operator on $\H$ if $\eps>0$. If $\eps=0$, this is still true since
the conditions $-\Delta u\in H$ and $u\in H^{3/2}(D)$ imply that $\partial_{\bf n}u\in H_\Gamma$ by \cite[Thm.~2.27]{BG87}.
Note that $\C_\eps$ is the strong formulation of the linear component of the operator $\A_\eps$ on $\H$.

\begin{lem}[Maximal monotonicity]
  \label{lm_Ceps}
  The operator $\C_\eps$ is maximal monotone on $\H$ and, consequently, its resolvent 
  $(I+\delta\C_\eps)^{-1}:\H\rarr\H$ is a linear contraction\footnote{Throughout the Appendix,
  by the term "contraction" we mean a 1-Lipschitz continuous function, i.e.~a non expansive operator.}
  for every $\delta>0$.
\end{lem}
\begin{proof}
  It is immediate to see that $\C_\eps$ is monotone. Hence, we only have to check that $R(I+\C_\eps)=\H$.
  Let then $(f,g)\in\H$: since $\C_\eps$ is coercive on $\V_\eps$, by the Lax-Milgram lemma there exists a couple 
  $(u,v)\in \V_\eps$ such that 
  \beq
  \label{var}
  \int_Du\varphi + \int_\Gamma v\psi + \int_D\nabla u\cdot\nabla \varphi +
  \eps\int_\Gamma\nabla_\Gamma v\cdot\nabla_\Gamma \psi =
  \int_Df\varphi + \int_\Gamma g\psi \quad\forall\,(\varphi,\psi)\in\V_\eps\,.
  \eeq
  Taking $\varphi\in C^\infty_c(D)$ and $\psi=0$ in the previous variational formulation, one 
  can see that $u$ satisfies in $D$ the following PDE (in the sense of distributions on $D$)
  \beq\label{pde1}
  u-\Delta u = f\,.
  \eeq
  We recover immediately that also $-\Delta u \in H$ by difference, and the previous 
  equation holds in $H$. Moreover, for any $\psi\in C^\infty(\Gamma)$,
  since the trace operator $\tau$ has a right-inverse which is linear and continuous from $H^{k-1/2}(\Gamma)$ to $H^k(D)$
  for every $k\geq1$ (see \cite[Thm.~2.24]{BG87}), thanks to the usual Sobolev embeddings results 
  there is $\varphi \in C^\infty(\overline{D})$ such that $\psi=\tau\varphi$:
  choosing $(\varphi, \psi)$ in \eqref{var}, integrating by parts on $D$ and taking \eqref{pde1} into account,
  we see that $v$ satisfies (in the sense of distributions on the boundary $\Gamma$)
  \beq\label{pde2}
  v + \partial_{\bf n}u -\eps\Delta_\Gamma v = g\,.
  \eeq
  Assume now $\eps>0$. By definition of $\V_\eps$, we know that $v=\tau u\in H^1(\Gamma)$:
  hence, thanks to \cite[Thm.~3.2]{BG87} we also deduce that $u\in H^{3/2}(D)$.
  Using the result contained in \cite[Thm.~2.27]{BG87} and the facts that 
  $u\in H^{3/2}(D)$ and $-\Delta u\in H$,  we deduce that $\partial_{\bf n}u\in H_\Gamma$ as well.
  By comparison in \eqref{pde2}
  we have that $v-\eps\Delta_\Gamma v\in H_\Gamma$, so that 
  $v\in H^2(\Gamma)$ by elliptic regularity on the boundary and \eqref{pde2} holds in $H_\Gamma$.
  Finally, since $\Delta u\in H$ and $v\in H^2(\Gamma)\embed H^{3/2}(\Gamma)$, thanks again to 
  \cite[Thm.~3.2]{BG87} we deduce that $u\in H^2(D)$. Hence, $(u,v)\in D(\C_\eps)$ and this completes the proof in the case $\eps>0$;
  moreover, note that the two results \cite[Thm.~2.27 and 3.2]{BG87} that we have used
  ensure also that 
  \[
  \l|(u,v)\r|_{H^2(D)\times H^2(\Gamma)}\lesssim \l|(f,g)\r|_\H\,.
  \]
  Assume $\eps=0$. By definition of $\V_\eps$, now we have $u\in H^1(D)$ and $v\in H^{1/2}(\Gamma)$.
  Moreover, by comparison in \eqref{pde2} we have $\partial_{\bf n}u \in H_\Gamma$:
  thanks to \cite[Thm.~3.2]{BG87} and the facts that $\Delta u\in H$ and $\partial_{\bf n}u\in H_\Gamma$,
  we deduce that $u\in H^{3/2}(D)$, and consequently also $v\in H^1(\Gamma)$. Hence, $(u,v)\in D(\C_\eps)$ and the proof is complete.
  As before, owing to the results \cite[Thm.~2.27 and 3.2]{BG87} we also have that 
  \[
  \l|(u,v)\r|_{H^{3/2}(D)\times H^1(\Gamma)}\lesssim \l|(f,g)\r|_\H\,.\qedhere
  \]
\end{proof}

\begin{rmk}
  Note that the first part of the proof of the previous lemma ensures that 
  $\C_\eps$ can be extended to a linear operator $\C_\eps:\V_\eps\rarr\V_\eps^*$,
  given by the left-hand side of \eqref{var}, which is still maximal monotone in the sense 
  of Minty--Browder theory (see \cite[Ch.~2]{barbu_monot}).
\end{rmk}

\begin{lem}[Regularity]
  \label{lm_Ceps2}
  Let $\delta>0$ and $k\in\En$. Then,
  \beq
    \label{regular}
    (I+\delta\C_\eps)^{-1}\in\begin{cases}
    \cL\left(H^{k}(D)\times H^k(\Gamma), H^{k+2}(D)\times H^{k+2}(\Gamma)\right) \quad&\text{if } \eps>0\,,\\
    \cL\left(H^{k}(D)\times H^k(\Gamma), H^{k+3/2}(D)\times H^{k+1}(\Gamma)\right) &\text{if } \eps=0\,.
    \end{cases}
  \eeq
\end{lem}
\begin{proof}
  It is not restrictive to assume that $\delta=1$.
  In Lemma \ref{lm_Ceps} we have already proved the case $k=0$:
  let us only show the case $k=1$, since for a general $k$ it follows by induction.
  Let $(f,g)\in H^1(D)\times H^1(\Gamma)$ and
  $(u,v)=(I+\C_\eps)^{-1}(f,g) \in D(\C_\eps)$. \\
  If $\eps>0$, by \eqref{pde1} we have
  $u-\Delta u\in H^1(D)$: since by definition of $D(\C_\eps)$ we also know that $v=\tau u\in H^2(\Gamma)$, it follows from
  \cite[Thm.~3.2]{BG87} that $u\in H^{5/2}(D)$. Consequently, $\partial_{\bf n}u\in H^1(\Gamma)$:
  hence, by comparison in \eqref{pde2} we have $v-\eps\Delta_{\Gamma}v\in H^1(\Gamma)$.
  By elliptic regularity on the boundary we deduce that $v\in H^3(\Gamma)$. 
  Combining this information with the fact that $u-\Delta u \in H^1(D)$, again by \cite[Thm.~3.2]{BG87}
  we also have $u\in H^3(D)$.
  Taking into account \cite[Thm.~2.27 and 3.2]{BG87}, it is clear also that 
  \[
  \l|(u,v)\r|_{H^{3}(D)\times H^3(\Gamma)}\lesssim \l|(f,g)\r|_{H^1(D)\times H^1(\Gamma)}\,.
  \]
  If $\eps=0$, since $v\in H^1(\Gamma)$, by difference in \eqref{pde2} we have $\partial_{\bf n}u\in H^1(\Gamma)$;
  this information, together with the fact that $-\Delta u\in H^1(D)$ by difference in \eqref{pde1}, implies that 
  $u\in H^{5/2}(D)$. Consequently, $v=\tau u \in H^2(\Gamma)$. Finally, again by
  \cite[Thm.~2.27 and 3.2]{BG87}, we have that 
  \[
  \l|(u,v)\r|_{H^{5/2}(D)\times H^2(\Gamma)}\lesssim \l|(f,g)\r|_{H^1(D)\times H^1(\Gamma)}\,.\qedhere
  \]
\end{proof}

\begin{lem}[Extension to $L^1$]
  \label{lm_Ceps3}
    Let $\delta>0$. The resolvent $(I+\delta\C_\eps)^{-1}$ can be uniquely extended to a linear contraction
    from $L^1(D)\times L^1(\Gamma)$ to itself. Moreover, one has that 
    $(I+\delta\C_\eps)^{-1}\in\cL(L^1(D)\times L^1(\Gamma), W^{1,q}(D)\times W^{1-1/q,q}(\Gamma))$
    for every $q\in[1,\frac{N}{N-1})$. 
\end{lem}
\begin{proof}
 Since $\delta>0$ is fixed throughout the proof, we do not use notation for
 the dependence on $\delta$ of the quantities that we introduce.
 Given $(f,g)\in L^1(D)\times L^1(\Gamma)$,
 let us consider $\{(f_n,g_n)\}_{n\in\En}\subseteq\H$ such that $(f_n,g_n)\rarr(f,g)$
 in $L^1(D)\times L^1(\Gamma)$ as $n\rarr\infty$, and $(u_n,v_n):=(I+\delta\C_\eps)^{-1}(f_n,g_n)$.
 Moreover, let
 $\{\rho_k\}_{k\in\En}$ be a sequence of smooth Lipschitz-continuous increasing functions on $\Ar$ approximating
  pointwise the maximal monotone graph 
  \[
  \text{sign}:\Ar\rarr2^\Ar\,, \quad \text{sign}(x):=
  \begin{cases}
  \frac x{|x|} \quad&\text{if } x\neq0\,,\\
  [-1,1] &\text{if } x=0\,.
  \end{cases}
  \]
  For example, one can take $\rho_k(x)=\tanh(kx)$, $x\in\Ar$. 
  Now, we consider equation \eqref{var} with respect to $n,m\in\En$, take the difference and 
  test by $(\rho_k(u_n-u_m), \rho_k(v_n-v_m))\in\V_\eps$, obtaining
  \[
  \begin{split}
  \int_D(u_n-u_m)&\rho_k(u_n-u_m) + \int_\Gamma (v_n-v_m)\rho_k(v_n-v_m) \\
  &\qquad\quad+ \delta\int_D\rho_k'(u_n-u_m)|\nabla (u_n-u_m)|^2 + 
  \eps\delta\int_\Gamma\rho_k'(v_n-v_m)|\nabla_\Gamma (v_n-v_m)|^2\\
  &=\int_D(f_n-f_m)\rho_k(u_n-u_m)+\int_\Gamma (g_n-g_m)\rho_k(v_n-v_m)\,.
  \end{split}
  \]
  Using monotonicity and the fact that $|\rho_k|\leq1$, letting $k\rarr\infty$, 
  by the dominated convergence theorem it is immediate to see that
  \[
  \int_D|u_n-u_m|+\int_\Gamma|v_n-v_m|\leq \int_D|f_n-f_m|+\int_\Gamma|g_n-g_m|\,:
  \]
  since $(f_n,g_n)\rarr(f,g)$ in $L^1(D)\times L^1(\Gamma)$, this implies that $\{(u_n,v_n)\}_{n\in\En}$ is Cauchy in
   $L^1(D)\times L^1(\Gamma)$.
  Hence, there is $(u,v)\in L^1(D)\times L^1(\Gamma)$ (which is independent of the approximating sequence
   $\{(f_n,g_n)\}_{n\in\En}$) such that 
  \[
  (u_n,v_n)\rarr(u,v) \quad\text{in } L^1(D)\times L^1(\Gamma)\,.
  \]
  This proves that $(I+\delta\C_\eps)^{-1}$ extends uniquely to a linear contraction on $L^1(D)\times L^1(\Gamma)$
  and the first part of the lemma is proved.\\
  Let us focus on the second part.
  First of all, we need to prove an auxiliary result, which is a generalization
  of the classical elliptic regularity theorems by Stampacchia (see \cite{stamp}). Namely,
  for every $p>N$ and $h_0,\ldots,h_N\in L^p(D)$, by the Lax-Milgram lemma there is 
  a weak solution $(z,w)\in\V_\eps$
  such that 
  \beq
  \label{var_aux}
  \int_Dz\varphi+ \int_\Gamma w\psi + \delta\int_D\nabla z\cdot\nabla\varphi + 
  \eps\delta\int_\Gamma\nabla_\Gamma w\cdot\nabla_\Gamma\psi=
  \int_Dh_0\varphi +\sum_{i=1}^N\int_Dh_i\frac{\partial \varphi}{\partial x_i}
  \eeq
  for every $(\varphi,\psi)\in\V_\eps$.
  Let us prove that
  $(z,w)\in L^\infty(D)\times L^\infty(\Gamma)$
  and that there exists $C>0$ such that
  \[
  \l|w\r|_{L^\infty(\Gamma)}\leq\l|z\r|_{L^\infty(D)}\leq C\sum_{i=0}^N\l|h_i\r|_{L^p(D)}\,.
  \]
  For every $k\in\En$, we introduce the Lipschitz function
  \[
  G_k:\Ar\rarr\Ar\,, \qquad G_k(t):=\begin{cases}
  t+k \quad&\text{if } t<-k\,,\\
  0 &\text{if } -k\leq t\leq k\,,\\
  t-k &\text{if } t>k\,.
  \end{cases}
  \]
  Testing \eqref{var_aux} by $(G_k(z), G_k(w))\in\V_\eps$, setting $A_k:=\{|z|\geq k\}\subseteq D$ we get
  \[
  \int_DG_k(z)z+\int_\Gamma G_k(w)w+\delta\int_{A_k}|\nabla G_k(z)|^2 + \eps\delta\int_{|w|\geq k}|\nabla_\Gamma w|^2=
  \int_{A_k}h_0G_k(z)+\sum_{i=0}^N\int_{A_k}h_i \frac{\partial z}{\partial x_i}\,.
  \]
  Using the Young inequality, the fact that $|G_k(z)|\leq|z|$ and the monotonicity of $G_k$, we deduce that 
  \[
  \l|G_k(z)\r|^2_H+
  \frac{\delta}{2}\int_{A_k}|\nabla G_k(z)|^2 \leq 
  \int_{A_k} h_0G_k(z) + \frac12\sum_{i=1}^N\int_{A_k}|h_i|^2\,,
  \]
  which can be rewritten as
  \[
 \l|G_k(z)\r|_{H^1(D)}^2\leq \max\left\{1,\frac2\delta\right\}\int_{A_k}h_0 G_k(z) +
  \max\left\{\frac12,\frac1\delta\right\} \sum_{i=1}^N\int_{A_k}|h_i|^2\,.
  \]
 Let us consider first the case $N\geq3$.
  If we set $2^*:=\frac{2N}{N-2}$ and $2_*:=\frac{2^*}{2^*-1}=\frac{2N}{N+2}$, 
  the Sobolev embedding $H^1(D)\embed L^{2^*}(D)$
  on the left-hand side and the H\"older inequality on the right-hand side yield
  \[
  \left(\int_{A_k}{|G_k(z)|^{2^*}}\right)^{\frac{2}{2^*}}\leq 
  C\left(\int_{A_k}|h_0|^{2_*}\right)^\frac{1}{2_*}\left(\int_{A_k}|G_k(z)|^{2^*}\right)^\frac{1}{2^*}+ C\sum_{i=1}^N\int_{A_k}|h_i|^2
  \]
  for a positive constant $C$, from which, thanks to the Young inequality we have
   \[
  \left(\int_{A_k}{|G_k(z)|^{2^*}}\right)^{\frac{2}{2^*}}\leq 
  C\left(\int_{A_k}|h_0|^{2_*}\right)^\frac{2}{2_*}+ 2C\sum_{i=1}^N\int_{A_k}|h_i|^2\,.
  \]
  Using the fact that $h_i\in L^p(D)$ for $i=0,\ldots,N$, that $p>2_*$ (since $p>N$) 
  we deduce
  \[
  \left(\int_{A_k}{|G_k(z)|^{2^*}}\right)^{\frac{2}{2^*}}\leq
  2C\left[\l|h_0\r|^2_{L^p(D)}|A_k|^{\frac{2}{2_*}-\frac2p}
  +|A_k|^{1-\frac2p}\sum_{i=1}^N\l|h_i\r|^2_{L^p(D)}\right]\,.
  \]
  Now, for every $h>k$ we have $A_h\subseteq A_k$ and $G_k(u)\geq h-k$ on $A_h$ so that 
  \[
  (h-k)^{2}|A_h|^{\frac{2}{2^*}}\leq 2C\sum_{i=0}^N\l|h_i\r|^2_{L^p(D)}\left(|A_k|^{\frac{2}{2_*}-\frac2p}+|A_k|^{1-\frac2p}\right)\,.
  \]
  Renominating the constant $C$, since it is not restrictive to assume that $|A_k|<1$, it follows
  \[
  |A_k|\leq C\left(\sum_{i=0}^N\l|h_i\r|_{L^p}^2\right)^\frac{2^*}{2}\frac{|A_k|^\alpha}{(h-k)^{2^*}}\,, \qquad
  \alpha:=\frac{2^*}{2}\min\left\{\frac{2}{2_*}-\frac{2}{p}, 1-\frac{2}{p}\right\}\,.
  \]
  Now, using the fact that $p>N$, it is a standard matter to see that
  \[
  \alpha=\frac{2^*}{2}\min\left\{\frac{N+2}{N}-\frac{2}{p}, 1-\frac{2}{p}\right\}=\frac{N}{N-2}\left(1-\frac{2}{p}\right)=
  \frac{1-2/p}{1-2/N}>1\,.
  \]
 If $N=2$, then we know that $H^1(D)\embed L^r(D)$ for all $r\in[1,+\infty)$: using this fact,
  we repeat the same argument replacing $2_*$ and $2^*$
  by an arbitrary $q\in(1,2)$ and its conjugate exponent $q'=\frac{q}{q-1}$, respectively.
  With such a choice, the same computations yield
  \[
  \alpha=\frac{q'}{2}\min\left\{\frac{2}{q}-\frac2p, 1-\frac2p\right\} = 
  \frac{q}{2(q-1)}\left(1-\frac2p\right)=\frac{q(p-2)}{2p(q-1)}\,.
  \]
  It is easily seen that $\alpha>1$ if and only if $q<\frac{2p}{p+2}$:
  since the fact that $p>2$ implies that $\frac{2p}{p+2}\in(1,2)$,
  we can choose $q\in(1,\frac{2p}{p+2})$, getting
  $\alpha>1$ also in the case $N=2$, as desired.
  By \cite[Lem.~4.1]{stamp}, we can conclude that $z\in L^\infty(D)$ and $\l|z\r|_{L^\infty(D)}\leq C\sum_{i=0}^N\l|h_i\r|_{L^p(D)}$,
  suitably renominating the positive constant $C$. Moreover, since $w=\tau z$, we also have that $w\in L^\infty(\Gamma)$ and
  $\l|w\r|_{L^\infty(\Gamma)}\leq\l|z\r|_{L^\infty(D)}$.\\
  We are now ready to complete the proof of the lemma. Testing \eqref{var_aux} by $(u_n,v_n)$,
  recalling the definition of $(u_n,v_n)$ we have
  \[
  \int_Dh_0u_n+\sum_{i=1}^N\int_Dh_i\frac{\partial u_n}{\partial x_i} =
  \int_Df_nz+\int_\Gamma g_nw
  \leq C\sum_{i=0}^N\l|h_i\r|_{L^p(D)}
  \left(\l|(f_n\r|_{L^1(D)}+\l|g_n\r|_{L^1(\Gamma)}\right)\,;
  \]
  taking into account that  $h_0,\ldots,h_N\in L^p(D)$ are arbitrary, we deduce that
  \[
  \l|\left(u_n,  \frac{\partial u_n}{\partial x_1}, \ldots, \frac{\partial u_n}{\partial x_N}\right)\r|_{L^q(D)^{N+1}}\leq
  C\l|(f_n,g_n)\r|_{L^1(D)\times L^1(\Gamma)}\,,
  \]
  where $q:=\frac{p}{p-1}$ is the conjugate exponent of $p$. Since
  $(f_n,g_n)\rarr(f,g)$ in $L^1(D)\times L^1(\Gamma)$, recalling that the operator $\C_\eps$ is linear and $p>N$, we have that
  $u_n\rarr u$ in $W^{1,q}(D)$ for every $q\in[1,\frac{N}{N-1})$, and consequently 
  $v_n\rarr v$ in $W^{1-1/q,q}(\Gamma)$. This ensures that $v=\tau u$;
  moreover,  letting $n\rarr\infty$ we have
  $\l|(u,v)\r|_{W^{1,q}(D)\times W^{1-1/q,q}(\Gamma)}\leq C\l|(f,g)\r|_{L^1(D)\times L^1(\Gamma)}$, 
  from which the thesis follows.
\end{proof}

\begin{lem}[Extension to $L^q$, $q>1$]
  \label{lm_Ceps4}
    Let $\delta>0$ and $q\in[1,\frac{N}{N-1})$. Then the resolvent $(I+\delta\C_\eps)^{-1}$ can be uniquely extended to a linear contraction
    from $L^q(D)\times L^q(\Gamma)$ to itself. Moreover, for every $k\in\En$, one has that
    \[
    (I+\delta\C_\eps)^{-1}\in
    \begin{cases}
    \cL(W^{k,q}(D)\times W^{k,q}(\Gamma), W^{k+2,q}(D)\times W^{k+2-1/q,q}(\Gamma)) \quad&\text{if } \eps>0\,,\\
    \cL(W^{k,q}(D)\times W^{k,q}(\Gamma), W^{k+1,q}(D)\times W^{k+1-1/q,q}(\Gamma)) \quad&\text{if } \eps=0\,.
    \end{cases}
    \] 
\end{lem}
\begin{proof}
  The fact that $(I+\delta\C_\eps)^{-1}$ can be extended to a contraction on $L^q(D)\times L^q(\Gamma)$
  can be showed in exactly the same way as in the proof of Lemma \ref{lm_Ceps3}: the only difference is the choice of 
  $\{\rho_k\}_{k\in\En}$. Here, one should take $\rho_k(t):\Ar\rarr\Ar$ smooth, increasing, Lipschitz continuous
  such that $\rho_k(0)=0$ and $\rho_k(t)=|t|^{q-2}t$ if $|t|\geq\frac1k$\,.\\
  Let us focus on the regularity result. We only show the case $k=0$, since
  one can easily generalize by induction to any $k\in\En$.
  Let then $(f,g)\in L^{q}(D)\times L^{q}(\Gamma)$ and let us consider $(u,v):=(I+\delta\C_\eps)^{-1}(f,g)$.
  By Lemma \ref{lm_Ceps3} we have that $u\in W^{1,q}(D)$, $v\in W^{1-1/q,q}(\Gamma)$, $v=\tau u$ and
  $u-\delta\Delta u =f$ in the sense of distributions on $D$. Hence, owing to \cite[Thm.~2.27]{BG87}
  we deduce that $\partial_{\bf n}u\in W^{-1/q,q}(\Gamma)$, so that we can write $v+\partial_{\bf n} u-\delta\Delta_\Gamma v=g$
  in the sense of distributions on $\Gamma$.
  If $\eps>0$, by elliptic regularity on the boundary we deduce that $v\in W^{2-1/q,q}(\Gamma)$: consequently, thanks to 
  \cite[Thm.~3.2]{BG87}, we infer that $u\in W^{2,q}(D)$.
  If $\eps=0$, we have by difference that $\partial_{\bf n}u\in L^q(\Gamma)\embed W^{-1/2q,q}(\Gamma)$:
  hence, the result \cite[Thm.~3.2]{BG87}
  ensures that $u\in W^{1+1/2q,q}(D)\embed W^{1,q}(D)$, and consequently $v\in W^{1-1/q,q}(\Gamma)$.
\end{proof}

\begin{cor}[Ultracontractivity]
\label{cor_Ceps}
  There exists $m\in\En$ such that, for every $\delta>0$,
  \[
  (I+\delta\C_\eps)^{-m}\in\cL\left(L^1(D)\times L^1(\Gamma), L^\infty(D)\times L^\infty(\Gamma)\right)\,.
  \]
\end{cor}
\begin{proof}
  It easily follows from Lemmas \ref{lm_Ceps3}--\ref{lm_Ceps4} and the Sobolev embeddings theorems.
\end{proof}

\begin{lem}[Asymptotics as $\delta\searrow0$]
  \label{lm_Ceps5}
  Let $\eps\geq0$ and $(u_\delta,v_\delta):=(I+\delta\C_\eps)^{-1}(f,g)$ for any $(f,g)$ for which it makes sense. 
  Then, as $\delta\searrow0$, we have
  \begin{align*}
  (u_\delta, v_\delta)\rarr(f,g) \quad&\text{in } L^1(D)\times L^1(\Gamma) &&\text{if}\qquad (f,g)\in L^1(D)\times L^1(\Gamma)\,,\\
  (u_\delta, v_\delta)\rarr(f,g) \quad&\text{in } \H &&\text{if}\qquad (f,g)\in \H\,,\\
  (u_\delta, v_\delta)\rarr(f,g) \quad&\text{in } \V_\eps &&\text{if}\qquad (f,g)\in \V_\eps\,.
  \end{align*}
\end{lem}
\begin{proof}
  We start with the case $(f,g)\in\H$: testing \eqref{var} by $(u_\delta, v_\delta)$ and
  using the Young inequality we easily deduce that
  \beq
  \label{lim1}
  \frac12\l|u_\delta\r|_H^2 + \frac12\l|v_\delta\r|_{H_\Gamma}^2 + \delta\l|\nabla u_\delta\r|_H^2 + \eps\delta\l|\nabla_\Gamma v_\delta\r|_{H_\Gamma}^2
  \leq \frac12\l|f\r|_H^2 + \frac12\l|g\r|_{H_\Gamma}^2\,.
  \eeq
  It follows (for a subsequence, which we still denote by $\delta$) that
  \[
  (u_\delta, v_\delta)\rarrw(u,v) \quad\text{in } \H\,, \qquad
  \delta(u_\delta, v_\delta)\rarr0 \quad\text{in } \V_\eps\,,
  \]
  where by a standard density argument $(u,v)=(f,g)$.
  Moreover, we also have that 
  \[
  \limsup_{\delta\searrow0}\l|(u_\delta, v_\delta)\r|_\H\leq\l|(f,g)\r|_\H\,,
  \]
  which implies that $(u_\delta, v_\delta)\rarr(f,g)$ in $\H$ for the original sequence.\\
  If $(f,g)\in L^1(D)\times L^1(\Gamma)$, we introduce $\{(f_n,g_n)\}_{n\in\En}\subseteq\H$ such that
  $(f_n,g_n)\rarr(f,g)$ in $L^1(D)\times L^1(\Gamma)$ as $n\rarr\infty$: let $(u_{n,\delta},v_{n,\delta}):=(I+\delta\C_\eps)^{-1}(f_n,g_n)$.
  Using the fact that $(I+\delta\C_\eps)^{-1}$ is a contraction on $L^1(D)\times L^1(\Gamma)$ (see Lemma \ref{lm_Ceps3}), we have
  \[
  \begin{split}
  &\l|(u_\delta,v_\delta)-(f,g)\r|_{L^1(D)\times L^1(\Gamma)}\leq
  \l|(u_\delta,v_\delta)-(u_{n,\delta},v_{n,\delta})\r|_{L^1(D)\times L^1(\Gamma)}\\
  &\qquad\qquad\qquad\qquad+
  \l|(u_{n,\delta},v_{n,\delta})-(f_n,g_n)\r|_{L^1(D)\times L^1(\Gamma)}+
  \l|(f_n,g_n)-(f,g)\r|_{L^1(D)\times L^1(\Gamma)}\\
  &\qquad\leq
  2\l|(f,g)-(f_n,g_n)\r|_{L^1(D)\times L^1(\Gamma)}+
  C\l|(u_{n,\delta},v_{n,\delta})-(f_n,g_n)\r|_\H
  \end{split}
  \]
  for a positive constant $C$ independent of $n$ and $\delta$.
  Now, for any $\eta>0$, there is $n\in\En$ such that the first term on the right-hand side of the previous 
  expression is controlled by $\eta$: for such an $n$, thanks to what we have already proved, there is $\delta$ such that
  the second term is less or equal than $\eta$. Hence, the right-hand side can be made smaller than $2\eta$ and the claim is proved.\\
  Finally, let $(f,g)\in\V_\eps$: for what we have already proved, we know that $(u_\delta, v_\delta)\rarr(f,g)$ in $\H$
  as $\delta\searrow0$. Moreover, we have
  \[
  (u_\delta,v_\delta)+\delta\C_\eps(u_\delta,v_\delta)=(f,g) \quad\text{in } \H\,:
  \]
  taking the scalar product in $\H$ with $\C_\eps(u_\delta,v_\delta)$ in the previous expression,
  using the fact that $(f,g)\in\V_\eps$ and integrating by parts we get
  \[
  \int_D|\nabla u_\delta|^2 + \eps\int_\Gamma|\nabla_\Gamma v_\delta|^2 + \delta\l|\C_\eps(u_\delta,v_\delta)\r|_\H^2=
  \int_D\nabla f\cdot\nabla u_\delta + \eps\int_\Gamma\nabla_\Gamma g\cdot\nabla_\Gamma v_\delta\,.
  \]
  The Young inequality yields then
  \[
  \l|\nabla u_\delta\r|_H^2+\eps\l|\nabla_\Gamma v_\delta\r|^2_{H_\Gamma}
  +2 \delta\l|\C_\eps(u_\delta,v_\delta)\r|_\H^2\leq
  \l|\nabla f\r|_H^2+\eps\l|\nabla_\Gamma g\r|^2_{H_\Gamma}\,,
  \]
  which together with \eqref{lim1} implies that 
  \[
  |(u_\delta,v_\delta)|^2_{\V_\eps}\leq|(f,g)|^2_{\V_\eps}\,.
  \]
  We deduce that
  \[
  (u_\delta,v_\delta)\rarrw(f,g) \quad\text{in } \V_\eps\,, \qquad \limsup_{\delta\searrow0}|(u_\delta,v_\delta)|_{\V_\eps}\leq|(f,g)|_{\V_\eps}\,,
  \]
  from which $(u_\delta,v_\delta)\rarr(f,g)$ in $\V_\eps$ as well.
\end{proof}

\begin{lem}[Maximum principle]
  \label{lm_Ceps6}
  Let $c_1,c_2>0$ and 
  $(f,g)\in L^1(D)\times L^1(\Gamma)$ with $f\leq c_1$ and $g\leq c_2$ almost everywhere on $D$
  and $\Gamma$, respectively; if $(u,v):=(I+\delta\C_\eps)^{-1}(f,g)$ then
  \[
  u\leq\max\{c_1,c_2\} \quad\text{a.e.~on } D\,, \qquad
  v\leq\max\{c_1,c_2\} \quad\text{a.e.~on } \Gamma\,.
  \]
\end{lem}
\begin{proof}
  Setting $c:=\max\{c_1,c_2\}$, we introduce the Lipschitz function
  $\rho(t):=(t-c)_+$, $t\in\Ar$: testing the corresponding variational formulation \eqref{var} by $(\rho(u),\rho(v))\in\V_\eps$ we have
  \[
  \int_D{\rho(u)u}+\int_\Gamma\rho(v)v+\delta\int_D\rho'(u)|\nabla u|^2 + \eps\delta\int_\Gamma\rho'(v)|\nabla_\Gamma v|^2=
  \int_Df\rho(u) + \int_\Gamma g\rho(v)\,.
  \]
  Using the definition of $\rho$, monotonicity and the hypotheses on $f$ and $g$ we infer that 
  \[
  \int_D|\rho(u)|^2+\int_\Gamma|\rho(v)|^2\leq\int_D(f-c)\rho(u)+\int_\Gamma(g-c)\rho(v)\leq0\,,
  \]
  from which $\rho(u)=0$ and $\rho(v)=0$. Hence, $u\leq c$ and $v\leq c$ almost everywhere.
\end{proof}

We introduce the projections on the first and second component, respectively, as
\[
  p_1: L^1(D)\times L^1(\Gamma)\rarr L^1(D)\,, \qquad
  p_2: L^1(D)\times L^1(\Gamma)\rarr L^1(\Gamma)\,.
\]
Let now $\delta>0$: we set
\begin{gather*}
  J_{\eps,\delta}^1:=p_1\circ(I+\delta\C_\eps)^{-1} :L^1(D)\times L^1(\Gamma)\rarr L^1(D)\,,\\
  J_{\eps,\delta}^2:=p_2\circ(I+\delta\C_\eps)^{-1}:L^1(D)\times L^1(\Gamma)\rarr L^1(\Gamma)\,.
\end{gather*}
Owing to Lemma \ref{lm_Ceps3}, 
it is well-clear that $J_{\eps,\delta}^i$ is a linear continuous operator for $i=1,2$ and that
for every $(f,g)\in L^1(D)\times L^1(\Gamma)$ by linearity we have
\[
(I+\delta\C_\eps)^{-1}(f,g)=\left(J_{\eps,\delta}^1(f,g), J_{\eps,\delta}^2(f,g) \right)=
\left(J_{\eps,\delta}^1(f,0)+J_{\eps,\delta}^1(0,g), J_{\eps,\delta}^2(f,0)+J_{\eps,\delta}^2(0,g)\right)\,.
\]

\begin{lem}[Convexity inequality]
  \label{lm_Ceps7}
  Let $(f,g)\in L^1(D)\times L^1(\Gamma)$ and $\Phi,\Psi:\Ar\rarr[0,+\infty)$ two proper convex and lower semicontinuous
  functions with $(\Phi(f), \Psi(g))\in L^1(D)\times L^1(\Gamma)$. Then, for every $\delta>0$ we have that
  \begin{gather*}
  \Phi\left(J_{\eps,\delta}^1(f,0)\right)+\Psi\left(J_{\eps,\delta}^1(0,g)\right)\leq
  J_{\eps,\delta}^1\left(\Phi(f),\Psi(g)\right) \quad\text{a.e.~in } D\,,\\
  \Phi\left(J_{\eps,\delta}^2(f,0)\right)+
  \Psi\left(J_{\eps,\delta}^2(0,g)\right)\leq J_{\eps,\delta}^2\left(\Phi(f),\Psi(g)\right) \quad\text{a.e.~in } \Gamma\,.
  \end{gather*}
\end{lem}
\begin{proof}
  We introduce the operators
  \begin{gather*}
    L_{\eps,\delta}^1:L^1(D)\rarr L^1(D)\,, \qquad
    L_{\eps,\delta}^1(f):=J_{\eps,\delta}^1(f,0)\,, \quad f\in L^1(D)\,,\\
    G_{\eps,\delta}^1:L^1(\Gamma)\rarr L^1(D)\,, \qquad
    G_{\eps,\delta}^1(g):=J_{\eps,\delta}^1(0,g)\,, \quad g\in L^1(\Gamma)\,.
  \end{gather*}
  Then, by Lemma \ref{lm_Ceps3} it is a standard matter to see that $L_{\eps,\delta}^1$ and $G_{\eps,\delta}^1$
  are linear contractions. Moreover, Lemma \ref{lm_Ceps6} ensures
  that they are sub-markovian operators in the sense of \cite[Def.~3.1]{haase}: hence,
  the generalized Jensen inequality contained in \cite[Thm.~3.4]{haase} implies that
  a.e.~on $D$
  \[
  \Phi\left(J_{\eps,\delta}^1(f,0)\right)\leq J_{\eps,\delta}^1\left(\Phi(f),0\right)\,, \qquad
  \Psi\left(J_{\eps,\delta}^1(0,g)\right)\leq J_{\eps,\delta}^1\left(0,\Psi(g)\right)\,.
  \]
  The first thesis follows summing the two inequalities, while the second can be easily proved
  with the other (obvious) choice of $L^2_{\eps,\delta}$ and $G^2_{\eps,\delta}$.
\end{proof}

\begin{cor}
  \label{cor_Ceps2}
  For every $(f,g)\in\V_\eps$ and for every $(h,\ell)\in L^1(D)\times L^1(\Gamma)$ such that
  $j(f)+j^*(h)\in L^1(D)$ and $j_\Gamma(g)+j_\Gamma^*(\ell)\in L^1(\Gamma)$, the families
  \[
  \left\{J_{\eps,\delta}^1(f,g)J_{\eps,\delta}^1(h,\ell)\right\}_{\delta>0} \qquad\text{and}\qquad
  \left\{J_{\eps,\delta}^2(f,g)J_{\eps,\delta}^2(h,\ell)\right\}_{\delta>0}
  \]
  are uniformly integrable on $D$ and $\Gamma$, respectively.
\end{cor}
\begin{proof}
  Using linearity, Young's inequality, the symmetry of $j$, $j_\Gamma$
  and Lemma \ref{lm_Ceps7} we have
  \[
  \begin{split}
  &\pm J_{\eps,\delta}^1(f,g)J_{\eps,\delta}^1(h,\ell)\leq
  \left(\pm J_{\eps,\delta}^1(f,0)\pm J_{\eps,\delta}^1(0,g)\right)
  \left(J_{\eps,\delta}^1(h,0)+J_{\eps,\delta}^1(0,\ell)\right)\\
  &= \pm J_{\eps,\delta}^1(f,0)J_{\eps,\delta}^1(h,0)
  \pm J_{\eps,\delta}^1(0,g)J_{\eps,\delta}^1(0,\ell)
  \pm J_{\eps,\delta}^1(0,g)J_{\eps,\delta}^1(h,0)
  \pm J_{\eps,\delta}^1(f,0)J_{\eps,\delta}^1(0,\ell)\\
  &\leq j\left(\pm J_{\eps,\delta}^1(f,0)\right)+j^*\left(J_{\eps,\delta}^1(h,0)\right)
  +j_\Gamma\left(\pm J_{\eps,\delta}^1(0,g)\right)+j_\Gamma^*\left(J_{\eps,\delta}^1(0,\ell)\right)\\
  &\qquad\qquad+j^*\left(J_{\eps,\delta}^1(h,0)\right)+j^*_\Gamma\left(J_{\eps,\delta}^1(0,\ell)\right)
  +j\left(\pm J_{\eps,\delta}^1(0,g)\right)+j_\Gamma\left(\pm J_{\eps,\delta}^1(f,0)\right)\\
  &\lesssim 1+
  j\left(J_{\eps,\delta}^1(f,0)\right)+j^*\left(J_{\eps,\delta}^1(h,0)\right)
  +j_\Gamma\left(J_{\eps,\delta}^1(0,g)\right)+j_\Gamma^*\left(J_{\eps,\delta}^1(0,\ell)\right)\\
  &\qquad\qquad+j^*\left(J_{\eps,\delta}^1(h,0)\right)+j^*_\Gamma\left(J_{\eps,\delta}^1(0,\ell)\right)
  +j\left(J_{\eps,\delta}^1(0,g)\right)+j_\Gamma\left(J_{\eps,\delta}^1(f,0)\right)\\
  &\leq1+
  J_{\eps,\delta}^1\left(j(f),j_\Gamma(g)\right)+2J_{\eps,\delta}^1\left(j^*(h),j_\Gamma^*(\ell)\right)
  +j\left(J_{\eps,\delta}^1(0,g)\right)+j_\Gamma\left(J_{\eps,\delta}^1(f,0)\right)\,.
  \end{split}
  \]
  Now, since $(j(f),j_\Gamma(g)), (j^*(h),j^*_\Gamma(\ell))\in L^1(D)\times L^1(\Gamma)$, 
  by Lemma \ref{lm_Ceps5} the sum of the first three terms on the right-hand side
  converge in $L^1(D)$ to $1+j(f)+2j^*(h)$ as $\delta\searrow0$.
  Hence, the first thesis follows if we are able to prove that 
  \[
  j\left(J_{\eps,\delta}^1(0,g)\right)+j_\Gamma\left(J_{\eps,\delta}^1(f,0)\right)
  \]
  is uniformly integrable on $D$. To this aim, we need to distinguish wether \eqref{H1}, \eqref{H2} or \eqref{H3}--\eqref{H3'}
  is in order. Firstly, if we assume hypothesis \eqref{H1}, the fact that $(j(f), j_\Gamma(g))\in L^1(D)\times L^1(\Gamma)$
  implies that also $(j_\Gamma(f), j(g))\in L^1(D)\times L^1(\Gamma)$: consequently, by Lemma \ref{lm_Ceps7},
  the two terms are bounded by $J^1_{\eps,\delta}\left(j_\Gamma(f), j(g)\right)$, which converges in
  $L^1(D)$ thanks to Lemma \ref{lm_Ceps5}.
  Secondly, let us assume \eqref{H2}. The facts that $j_\Gamma$ controls $j$ and $j_\Gamma(g)\in L^1(\Gamma)$
  imply that $j(g)\in L^1(\Gamma)$, so that by Lemma \ref{lm_Ceps7} 
  the first term is handled by $J^1_{\eps,\delta}\left(0, j(g)\right)$,
  which converges in $L^1(\Gamma)$ by Lemma \ref{lm_Ceps5}. Moreover, 
  $f\in H^1(D)$ and the Sobolev embeddings ensure that
  \[
  \begin{cases}
  H^1(D)\embed L^p(D) \quad\forall\,p\in[1,+\infty)\qquad&\text{if } N=2\,,\\
  H^1(D)\embed L^{\frac{2N}{N-2}}(D) &\text{if } N>2\,:
  \end{cases}
  \]
  hence, hypothesis \eqref{H2} implies that $j_\Gamma(f)\in L^1(D)$, so that the second term is
  bounded by $J^1_{\eps,\delta}\left(j_\Gamma(f),0\right)$, which converges in $L^1(D)$
  by Lemma \ref{lm_Ceps5}. Finally, let us assume (H3).
  Since $j$ controls $j_\Gamma$ and $j(f)\in L^1(D)$, we have also $j_\Gamma(f)\in L^1(D)$:
  hence, by Lemma \ref{lm_Ceps7}
  the first term is handled by $J^1_{\eps,\delta}\left(j_\Gamma(f),0\right)$, which converges in $L^1(D)$
  by Lemma \ref{lm_Ceps5}. Let us focus on the first term $j\left(J_{\eps,\delta}^1(0,g)\right)$.
  If $\eps>0$, we have $g\in H^1(\Gamma)$ and by the Sobolev embeddings (since $\Gamma$
  has dimension $N-1$)
  \[
  \begin{cases}
  H^1(\Gamma)\embed L^\infty(\Gamma) \qquad&\text{if } N=2\,,\\
  H^1(\Gamma)\embed L^p(\Gamma) \quad\forall\,p\in[1,+\infty) &\text{if } N=3\,,\\
  H^1(\Gamma)\embed L^{\frac{2(N-1)}{N-3}}(\Gamma) &\text{if } N>3\,.
  \end{cases}
  \]
  Hence, hypothesis \eqref{H3} ensures that $j(g)\in L^1(\Gamma)$, so that by Lemma \ref{lm_Ceps7} we have
  $j\left(J_{\eps,\delta}^1(0,g)\right)\leq J_{\eps,\delta}^1\left(0, j(g)\right)$, which converges in 
  $L^1(D)$ by Lemma \ref{lm_Ceps5}. Similarly, if $\eps=0$ then $g\in H^{1/2}(\Gamma)$
  and by the Sobolev embeddings we have
  \[
  \begin{cases}
  H^{1/2}(\Gamma)\embed L^p(\Gamma) \quad\forall\,p\in[1,+\infty) \qquad&\text{if } N=2\,,\\
  H^{1/2}(\Gamma)\embed L^\frac{2(N-1)}{N-2}(\Gamma) &\text{if } N>2\,.
  \end{cases}
  \]
  Consequently, \eqref{H3'} ensures again that $j(g)\in L^1(\Gamma)$, and we can conclude as in the case
  $\eps>0$.\\
  We have proved that $\pm J_{\eps,\delta}^1(f,g)J_{\eps,\delta}^1(h,\ell)$, hence
  also $|J_{\eps,\delta}^1(f,g)J_{\eps,\delta}^1(h,\ell)|$, is bounded by a family 
  which converges in $L^1(D)$ as $\delta\searrow0$, from which the uniform integrability follows.
  The argument for the family $\left\{J_{\eps,\delta}^2(f,g)J_{\eps,\delta}^2(h,\ell)\right\}_{\delta>0}$
  is exactly the same, and this completes the proof.
\end{proof}


\bibliography{ref}{}
\bibliographystyle{abbrv}

\end{document}